\documentclass{article}
\usepackage[utf8]{inputenc}
\usepackage{geometry}
\usepackage{tabularx,array}
\newcolumntype{M}[1]{>{\centering}m{#1}}
\usepackage{graphicx}
\usepackage{amssymb, amsmath, amsthm, mathtools}
\usepackage{epstopdf}
\usepackage{nicefrac}
\usepackage{enumerate}
\usepackage{hyperref, float}
\usepackage{color}
\usepackage{tikz}
\usepackage{pst-all}
\usepackage[alphabetic,abbrev,lite,backrefs]{amsrefs}
\usepackage{wrapfig}

\newtheorem{theorem}{Theorem}[section]

\newtheorem{cor}[theorem]{Corollary}
\newtheorem{lemma}[theorem]{Lemma}

\newtheorem{prop}[theorem]{Proposition}

\newtheorem*{thma}{Theorem A}{\bf}{\it}
\newtheorem*{thmb}{Theorem B}{\bf}{\it}
\newtheorem*{thmc}{Theorem C}{\bf}{\it}

\theoremstyle{definition}
\newtheorem{ex}[theorem]{Example}

\newtheorem{observation}[theorem]{Observation}
\newtheorem{question}[theorem]{Question}

\usepackage{dsfont}
\newcommand{\one}{{\mathds 1\!}} 
\newcommand{\BU}{\mathtt{BU}}
\DeclareMathOperator*{\Li}{\mathtt{Li}}
\DeclareMathOperator*{\Ls}{\mathtt{Ls}}
\newcommand{\Klim}{\operatorname*{K-lim}}
\newcommand{\ddt}{\frac{d}{dt}}

\newcommand{\norm}[1]{\|#1\|}
\newcommand{\bhor}{\partial_h}
\newcommand{\inv}{^{-1}}
\newcommand{\pD}{\text{\scalebox{-1}[1]{$\mathrm{P}$}\!\!$\mathrm{D}$}}

\newcommand{\HR}{H(\mathbb{R})}
\newcommand{\HHR}{H_{2n+1}(\mathbb{R})}
\newcommand{\g}{\mathfrak g}
\newcommand{\lv}{\langle}
\newcommand{\rv}{\rangle}

\newcommand{\Span}{\mathrm{span}}
 
\newcommand{\aff}{\mathrm{aff}}

\usepackage{relsize}
\newcommand{\N}{{\mathbb N}}

\newcommand{\R}{{\mathbb R}}

\newcommand{\Z}{{\mathbb Z}}

\newcommand{\cC}{{\mathcal C}}

\newcommand{\cF}{{\mathcal F}}

\newcommand{\cP}{{\mathcal P}}

\title{A metric boundary theory for Carnot groups}
\author{Nate Fisher}
\date{}

\begin{document}

\maketitle
\large

\begin{abstract}
In this paper, we study characteristics of horofunction boundaries of Carnot groups. In particular, we show that for Carnot groups, i.e., stratified nilpotent Lie groups equipped with certain left-invariant homogeneous metrics, all horofunctions are piecewise-defined using Pansu derivatives. For higher Heisenberg groups and filiform Lie groups, two families which generalize the standard 3-dimensional real Heisenberg group, we study the dimensions and topologies of their horofunction boundaries. In doing so, we find that filiform Lie groups of dimension $n\geq 8$ provide the first-known examples of Carnot groups $G$ whose horofunction boundaries are not of dimension $\dim(G) - 1$.
\end{abstract}

%\tableofcontents

\section{Introduction}

In geometric group theory, geometric topology, and dynamics, a great amount of work has been done to understand boundaries of groups and of metrics spaces. These boundary theories have proven incredibly useful, such as in the proof of Mostow's rigidity theorem about the geometry and topology of hyperbolic manifolds. Boundaries have also been used to understand algebraic splittings of groups and the asymptotic behavior of random walks.

Recently, the horofunction boundary,  a natural boundary arising from metric compactification, has been increasingly studied as a more suitable boundary for groups of mixed curvature. In particular, for nilpotent groups, Poisson boundaries are trivial, and visual boundaries have been shown to have trivial topologies \cite{DM61, azencott, DF-stars}. This work, building upon results of \cite{KN-cc, asympt-horobdry, FNG} among others, focuses on examining the horofunction boundaries of nilpotent Lie groups.

We recall the definition of the horofunction boundary. For a metric space $(X,d)$ with basepoint $x_0 \in X$, we embed $(X,d)$ into the space $\mathcal C(X)$ of continuous real-valued functions on $X$ via the map $\iota: x\mapsto d(x,\cdot) - d(x, x_0)$. We then take the closure under the compact-open topology and define the horofunction boundary to be $\bhor(X, d) := \overline{\iota(X)}\setminus \iota(X)$. See Section~\ref{sec:horobdry} for more detail.

In previous work with Sebastiano Nicolussi Golo, the author found that all horofunctions of the 3-dimensional real Heisenberg group $H(\R)$ when equipped with any sub-Finsler metric were piecewise-linear functions of coordinates coming from the first, or horizontal, layer of the Heisenberg grading. Combining results of \cite{FNG, KN-cc}, we also know that the horofunction boundaries of $\HR$ equipped with polygonal sub-Finsler metrics, which arise as the asymptotic cones of word metrics on $H(\Z)$, or with sub-Riemannian metrics are topologically two-dimensional. That is, the horofunction boundaries are of dimension $\dim\HR -1$. In this paper, we substantially widen the scope of groups under consideration to Carnot groups equipped with what we will call {\em layered sup norms}: homogeneous norms realized as $L^\infty$ combinations of norms on the layers. We ask how previous results generalize to these other Carnot groups:
\begin{enumerate}
\item Are horofunctions always piecewise-linear functions of the coordinates coming from the first layer of the grading? 
\item Is the horofunction boundary of a Carnot group $G$ always of dimension $\dim(G) - 1$?
\end{enumerate}

In this paper, we will say the horofunction boundary has \emph{expected dimension} if its dimension is one less than the dimension of the metric space. Indeed, that is what we see for normed vector spaces \cite{gutierrez, ji-schilling-polyhedral, karlsson-metz, kapovich-leeb-finsler, schilling-thesis, walsh-norm} and Carnot--Carath\'eodory metrics on the Heisenberg group \cite{KN-cc, FNG}.

A nilpotent Lie group $G$ is Carnot if it is stratified, i.e., if its Lie algebra satisfies a grading criterion, and if it is equipped with a left-invariant metric which, like the Euclidean distance, is homogeneous with respect to a family of dilations. See Sections~\ref{sec:gradings} and \ref{subsec:quasi-norms} for a more detailed definitions.
We will consider metrics on these stratified groups derived from homogeneous quasi-norms. These homogeneous quasi-norms are of interest to analysts, have been used to understand the large-scale geometry of nilpotent groups, and appear in the visual boundaries of symmetric spaces \cite{breuillard-loc-cpct-gps, dymarz-fisher-xie, fischer-ruzhansky, guivarch-quasinorm}. Specifically, in this paper we will consider metrics derived from layered sup quasi-norms, defined by taking the max of norms of projections of an element to each layer of the grading of $G$. In particular, we will consider {\em polysmooth} layered sup norms, where on each layer we take a norm which is either polyhedral or smooth, as defined in Section~\ref{subsec:quasi-norms}.

For the class of polysmooth layered sup norms on Carnot groups, we are able to answer the first question from above in the affirmative. 

\begin{thma}
If $G$ is a stratified group equipped with a layered sup norm built of polyhedral or smooth norms, every horofunction is piecewise-linear and expressed only in terms of coordinates of the first layer of the grading.
\end{thma}

To make progress towards answering the second question about the dimension of horofunction boundaries, we considered two infinite families of stratified groups which generalize the 3-dimensional Heisenberg group: 1) higher Heisenberg groups $H_{2n+1}(\R)$ of dimension $2n+1$ and step 2; and 2) filiform groups $L_n$ of the first kind of dimension $n$ and step $n-1$. For higher Heisenberg groups, we found that the horofunction boundary always has expected dimension. For certain exceptional layered sup norms which exhibit special symmetries, typically disjoint parts of the horofunction boundary have nonempty intersection. We will say that these norms have {\em non-separated} boundaries (see Section~\ref{sec:degenerate}), and the topology of the horofunction boundary in these exceptional cases can differ from the typical case.

\begin{thmb} The horofunction boundary of $\HHR$ equipped with a layered sup norm built of polyhedral or smooth norms is $2n$-dimensional. Except for non-separated norms, the horofunction boundary is homeomorphic to a $2n$-dimensional button pillow, that is, a $2n$-dimensional sphere with closed neighborhoods of the north and south poles identified.
\end{thmb}

On the other hand, for the family of filiform Lie groups, there is a threshold in dimension, and hence in nilpotency class, at which the horofunction boundary no longer has the expected dimension.

\begin{thmc} When equipped with a layered sup norm built of polyhedral or smooth norms, the horofunction boundary of the filiform group $L_n$ (of the first kind) has dimension $n-1$ for $2\leq n \leq 7$ and dimension strictly less than $n-1$ when $n\geq 8$.
\label{thm:thmc}
\end{thmc}

Theorem C provides the first known examples of Carnot groups whose horofunction boundaries are not of the expected dimension. This threshold at dimension 8 is curious and begs for further exploration into how the dimension of the boundary depends on the nilpotency class and the structure of the grading.

\subsection*{Acknowledgements}

The author would like to thank Tullia Dymarz for many helpful conversations and general guidance. He would also like to thank Moon Duchin, Anders Karlsson, and the anonymous referee for useful comments and suggestions regarding this work. The author is also grateful to Camilo Ramírez Maluendas for his invitation to visit the Universidad Nacional de Colombia in the spring of 2022, where significant progress was made on this project. This work was supported in part by NSF grant DMS-2230900.

\section{Background and notation}

We will start by surveying some background material on graded Lie groups, homogeneous metrics, and horofunction boundaries. For more details on graded Lie groups, Carnot groups, and homogeneous structures on nilpotent groups, we refer the reader to \cites{fischer-ruzhansky, ledonne-primer}.

%%%%%%%%%%%%%%%%%%%%%%%%%%%%%%%%%%%%%%%%%%%%%%%%%%%%%%%%%%%%%%%%%%%%%%%%%%%%%%%%%%%%%%%%%%%%

\subsection{Graded Lie algebras}
\label{sec:gradings}

Let $G$ be a finite-dimensional Lie group with Lie algebra $\g$ and Lie bracket $[\cdot,\cdot]:\g\times \g\to \g$. We say that $\g$ is \emph{gradable} if admits a {\em grading}, that is, a vector space decomposition of $\g$ into linear subspaces,
$\g =\mathlarger\oplus_{s>0} V_s$
such that all but finitely many of the subspaces $V_s$ are $\{0\}$ and such that 
\[
[V_s,V_r]:= \Span\{[v,w]:v\in V_s, w\in V_r\}\subseteq V_{s+r} \text{ for all } s, r > 0.
\] 

A Lie algebra equipped with a specific grading is called a {\em graded Lie algebra}. We call the subspace $V_s$ the {\em layer of degree $s$}. Note that that a gradable Lie algebra admits infinitely many gradings. Indeed, if $(V_s)_{s>0}$ is a grading of $\g$, then for any $r>0$ we get another grading $(W_s)_{s>0}$, where $W_s = V_{s/r}$. Using this scaling, we could, for example, scale any grading so that the first nonzero layer is the layer of degree 1, allowing for a simpler comparison of the two gradings.
We say that a Lie group is {\em graded} if it is simply connected and its Lie algebra is graded.

On any graded Lie algebra, we can define a family of dilations $\{\delta_t: \g \to \g\mid t>0\}$ where $\delta_t X = t^j X$ for any $X \in V_j$. We note that $\delta_t$ is a Lie algebra morphism---that is, $[\delta_t X, \delta_t Y]= \delta_t[X,Y]$. Additionally, we have $\delta_t\circ \delta_s = \delta_{ts}$.

Since the exponential map is a global diffeomorphism, by using exponential coordinates, we can define these dilations on the associated Lie group as well. See the next subsection for more on exponential coordinates.

We recall that the {\em lower central series} of a Lie algebra is a descending series of Lie subalgebras
\[
\g = \g_1 \supset \g_2 \supset \cdots \supset \g_n \supset \cdots,
\]
where $\g_1 = \g$ and where $\g_n = [\g, \g_{n-1}] = \Span\{[v,w]:v\in \g,\ w\in \g_{n-1}\}$. We say that a Lie algebra is {\em nilpotent of step $s$} if $\g_{s+1} = \{0\}$ while $\g_s \neq \{0\}$. One can show that any gradable Lie algebra is nilpotent for some $s$.

A {\em stratification} of a Lie algebra $\g$ is a particular type of grading of the form
$\g = V_1 \oplus V_2 \oplus \cdots \oplus V_s,$
where $[V_1,V_j]=V_{j+1}$ for $1 \leq j \leq s$, where $V_s \neq \{0\}$, and where $V_{s+1} = \{0\}$. We say that $\g$ is {\em stratifiable} if it admits a stratification, and $\g$ equipped with a stratification is a {\em stratified Lie algebra}. Equivalently, a graded Lie algebra is stratified if the first layer of the grading $V_1$ bracket-generates the entire Lie algebra. We say that a Lie group is {\em stratified} if it is simply connected and its Lie algebra is stratified.

While every gradable Lie algebra, and hence every stratifiable Lie algebra, is nilpotent, the converse is not true.

\begin{ex}{(Nilpotent Lie algebra which is not gradable)}
Consider the Lie algebra of dimension 7 generated by $X_1, \ldots, X_7$ whose only non-trivial bracket relations are
\[
[X_1, X_j] = X_{j+1} \text{ for } j = 2,\ldots 6, \quad [X_2, X_3] = X_6, \quad [X_2,X_4] = [X_5, X_2] = [X_3,X_4] = X_7.
\]
This Lie algebra is nilpotent but admits no grading.
\end{ex}

Additionally, it is worth noting that not every gradable Lie algebra is stratifiable. Indeed, we have the following example in dimension 5.

\begin{ex}{(Gradable Lie algebra which is not stratifiable)}
Consider the nilpotent Lie algebra of dimension 5 with basis $X_1, \ldots, X_5$, where the only non-trivial brackets are given by
\[
[X_1, X_j] = X_{j+1} \text{ for } 2\leq j \leq 4, \quad [X_2,X_3] = X_5.
\]
We can define a grading by putting each basis element in its own layer,
$\g = \langle X_1\rangle \oplus \langle X_2\rangle \oplus \langle X_3\rangle \oplus \langle X_4\rangle \oplus  \langle X_5\rangle.$ 
On the other hand, as an exercise, one can see that this Lie algebra does not admit a stratification.
\end{ex} 

%%%%%%%%%%%%%%%%%%%%%%%%%%%%%%%%%%%%%%%%%%%%%%%%%%%%%%%%%%%%%%%%%%%%%%%%%%%%%%%%%%%%%%%%%%%%

\subsection{Exponential coordinates and the Baker--Campbell--Hausdorff formula}\label{sec:exp-coords}

Let $\g$ be a nilpotent Lie algebra. On $\g$ we can define a group operation via the Baker--Campbell--Hausdorff formula
\begin{align*}
PQ &= \sum_{n=1}^\infty\frac{(-1)^{n-1}}{n} \sum_{\{s_j+r_j>0:j=1\dots n\}}
	\frac{ [P^{r_1}Q^{s_1}P^{r_2}Q^{s_2}  \cdots P^{r_n}Q^{s_n}] }
	{\sum_{j=1}^n (r_j+s_j) \prod_{i=1}^n r_i!s_i!  } \\
	&=P+Q+\frac12[P,Q]+ \frac1{12}([P,[P,Q]]-[Q,[P,Q]]) + \cdots ,
\end{align*}
where 
\[
[P^{r_1}Q^{s_1}P^{r_2}Q^{s_2}  \cdots P^{r_n}Q^{s_n}]
= \underbrace{[P,[P,\dots,}_{r_1\text{ times}}
\underbrace{[Q,[Q,\dots,}_{s_1\text{ times}}
\underbrace{[P,\dots}_{\dots}]\dots]]\dots]] .
\]
The sum in the formula above is  finite  because $\mathfrak g$ is nilpotent.

Using the exponential map $\exp: \g \to G$, which is a global diffeomorphism for nilpotent Lie algebras, we get an identification of the Lie algebra $\g$ and the simply connected nilpotent Lie group $G$. 
Once we fix a basis for the Lie algebra $\g$, through the exponential map we can identify $G$ with $\R^n$ with group multiplication determined by the Baker--Campbell--Hausdorff formula. These coordinates on $G$ are called {\em exponential coordinates of the first kind}.
In these coordinates, we note that $p^{-1}=-p$ for every $p\in G$ and that $e = \exp(0)$ is the neutral element of~$G$.

\subsubsection{Examples}
In this paper, we will look at two infinite families of stratified groups which generalize the 3-dimensional real Heisenberg group $\HR$. In exponential coordinates, we identify $\HR$ with $\R^3$ with group multiplication given by
\[
(x,y,z) (x', y', z') = \left(x + x', y + y', z + z' + \frac12(xy' - yx')\right).
\]

The first family which generalizes $\HR$ is the family of higher Heisenberg groups. We let $\HHR$ denote the real Heisenberg group of dimension $2n+1$ and of step 2. In exponential coordinates, $\HHR \cong \R^{2n+1}$, which we can think of as $\R^n \times \R^n \times \R$, where
\[
(x,y,z) (x', y', z') = \left(x + x', y + y', z + z' + \frac12(x\cdot y' - x'\cdot y)\right),
\]
and where $\cdot$ above represents the standard inner product on $\R^n$. On these Heisenberg groups, the standard stratification is
$\lv X_1, \ldots, X_n, Y_1, \ldots, Y_n\rv \oplus \lv Z\rv,$ 
with the associated family of dilations 
$\delta_t(x,y,z) = (tx, ty, t^2z).$

The second family of stratified groups generalizing $\HR$ is the family of filiform groups. The filiform group of the first kind $L_n$ is an $n$-dimensional stratifiable Lie group of step $n-1$. We can define $L_n$ by first defining its Lie algebra. Let $\mathfrak{l}_n$ be the $n$-dimensional Lie algebra generated by the vectors $X_1, X_2, \ldots, X_n$ where the only nontrivial bracket relations are
$[X_1,X_j] = X_{j+1},$ for $2\leq j \leq n-1.$
This Lie algebra can be stratified as 
$\g = \langle X_1, X_2\rangle \oplus \langle X_3\rangle \oplus \cdots \oplus \langle X_n\rangle.$
Note that the first layer of the stratification bracket-generates the whole Lie algebra. To understand the group multiplication in these filiform Lie groups, we can make use of the Baker--Campbell--Hausdorff formula. We note that $L_3 = \HR$, and for $L_4$, the 4-dimensional, step-3 filiform group, known as the Engel group, we run the computation. For two general vectors $P$ and $Q$ with
\[
P = xX_1 + yX_2 + zX_3 + wX_4, \quad\text{and}\quad Q = x'X_1 + y'X_2 + z'X_3 + w'X_4,
\]
we want to know what the product $PQ$ is. Since $\mathfrak{l}_4$ is nilpotent of step 4, we need only consider terms in the Baker--Campbell--Hausdorff formula with iterated brackets of length 3 or less. Thus we need only compute
\[
PQ = P+Q+\frac12[P,Q]+ \frac1{12}([P,[P,Q]]-[Q,[P,Q]]).
\] 
Using the bracket relations for $\mathfrak l_4$ given above, we find that 
\begin{align*}
PQ = \; &(x+x')X_1 + (y + y')X_2 
+\Bigl( z + z' + \frac12 (xy' - yx')\Bigr)X_3 \\
&+ \Bigl(w + w' + \frac12 (xz' - zx') + \frac1{12} (x-x')(xy'-yx')\Bigr)X_4.
\end{align*}
In exponential coordinates, then, the group multiplication in $L_4$ looks like
\begin{align*}(x,y,z,w)(x',y',z',w') = \Bigl( x+x', y+y'&, z+z'+\frac12(xy' - yx'), \\
& w + w' +\frac12(xz'-zx')+ \frac1{12}(x-x')(xy'-yx')\Bigr).
\end{align*}

We observe that the group multiplication in the first three coordinates is identical to that of the Heisenberg group $L_3$. Indeed, for all $k$, the group multiplication in the first $k$ coordinates of $L_{k+1}$ will be identical to that of $L_k$. Following the same method as above, one could determine the group multiplication for each of the filiform groups $L_n$, the main obstruction being the complicated nature of the Baker--Campbell--Hausdorff formula for Lie algebras of higher nilpotency class.
%%%%%%%%%%%%%%%%%%%%%%%%%%%%%%%%%%%%%%%%%%%%%%%%%%%%%%%%%%%%%%%%%%%%%%%%%%%%%%%%%%%%%%%%%%%%

\subsection{Homogeneous structures on graded groups}

Let $g$ be a graded group with the dilations $\delta_t$ associated to that grading. We will call the exponents of $t$ in the dilations of each coordinate the {\em weights}, and we will denote by $\nu_j$ the dilation weight associated to the coordinate $x_j$ in the group.

%By choosing an appropriate basis of the Lie algebra, we can record the grading or dilation structure of a Lie group in a diagonal matrix with the weights as the only nonzero entries. For example, in the previous paragraph we considered $\HR$ with the dilation structure $\begin{psmallmatrix} 1 & & \\ & \ell & \\ & & \ell +1\end{psmallmatrix}$.

A Lie group homomorphism between two graded groups $F:G\to G'$ is {\em homogeneous} if and only if $F\circ\delta_t = \delta_t'\circ F$ for all $t>0$, where $\delta'_t$ denotes the dilations in $G'$. A scalar function $f$ on $G\setminus {e}$ is homogeneous of degree $\nu$ if $f\circ \delta_t = t^\nu f$ for all $t>0$. For example, the coordinate function $x_j$ given by
$x = (x_1, \ldots, x_n) \mapsto x_j$ 
is homogeneous of degree $\nu_j$. For any nonzero multi-index $\alpha = (\alpha_1, \ldots, \alpha_n) \in \N_0^n\setminus \{\vec 0\}$, the function $x^\alpha := x_1^{\alpha_1}\cdots x_n^{\alpha_n}$ is homogeneous of degree $[\alpha]:= \nu_1\alpha_1 + \cdots + \nu_n\alpha_n$, called the {\em homogeneous degree} of the multi-index.

Next we state Proposition 3.1.24 from \cite{fischer-ruzhansky}, which tells us that the group multiplication we get from the Baker--Campbell--Hausdorff formula is polynomial in each coordinate and that the polynomial in the $j$th coordinate must have homogeneous degree equal to $\nu_j$.
\begin{prop}\label{prop:polynomials} (\cite{fischer-ruzhansky} Prop. 3.1.24).
For any $j = 1, \ldots, n,$ we have
\[
(xy)_j = x_j + y_j + \sum_{[\alpha] + [\beta] = \nu_j} c_{j,\alpha,\beta}x^\alpha y^\beta,
\]
where $\alpha$ and $\beta$ are multi-indices in $\N_0^n\setminus\{\vec 0\}$ and $c_{j, \alpha, \beta}$ are real coefficients.
\end{prop}

%%%%%%%%%%%%%%%%%%%%%%%%%%%%%%%%%%%%%%%%%%%%%%%%%%%%%%%%%%%%%%%%%%%%%%%%%%%%%%%%%%%%%%%%%%%%

\subsection{Homogeneous distances and quasi-norms}\label{subsec:quasi-norms}

A \emph{homogeneous distance} on $G$ is a distance function $d$ that is left-invariant and 1-homogeneous with respect to dilations, i.e., 
\begin{enumerate}
\item[(i)]
$d(gx,gy)=d(x,y)$ for all $g,x,y\in G$;
\item[(ii)]
$d(\delta_t x,\delta_t y) = t\cdot d(x,y)$ for all $x,y\in G$ and all $t>0$.
\end{enumerate}

When a stratified group $G$ is endowed with a homogeneous distance $d$, we call the metric Lie group $(G,d)$ a \emph{Carnot group}.
Homogeneous distances induce the manifold topology of  $G$, see \cite[Proposition 2.26]{MR3943324}, and are bi-Lipschitz equivalent to each other.
Every homogeneous distance defines a \emph{homogeneous norm}  $d_e(\cdot): G \to [0, \infty)$, where $d_e(p) = d(e, p)$ and $e$ is the identity element of $G$.

Relaxing the definition slightly, we can examine {\em homogeneous quasi-norms} $|\cdot|: G \to [0,\infty)$, continuous non-negative functions satisfying
\begin{enumerate}
\item $|x\inv| = |x|$ for all $x \in G$,
\item $|\delta_t x| = t|x|$ for all $x \in G$ and $t >0$,
\item $|x| = 0$ if and only if $x = e$.
\end{enumerate}

Note that, unlike a norm, a quasi-norm need not satisfy the triangle inequality. For example, on any graded Lie group, we can define the homogeneous $p$-quasi-norms and the homogeneous sup quasi-norm:
\begin{align*}
|(x_1, \ldots, x_n)|_p & := \left(\sum_{j=1}^n |x_j|^{p/\nu_j}\right)^{1/p}, \quad 0<p<\infty,\\
|(x_1, \ldots, x_n)|_\infty & := \max_{1\leq j\leq n}\left\{ |x_j|^{1/\nu_j}\right \}.
\end{align*}

Suppose that $G$ is a graded group. We equip each layer $V_j$ of the grading with an ordinary norm $\norm{\cdot}_{V_j}: V_j \cong \R^k \to [0, \infty)$, and let $\pi_j$ be the projection from $G$ to the coordinates in the $j$th layer of the stratification. We can define a homogeneous quasi-norm on $G$ via
\[
|x| := \max_{1\leq j \leq s} \left\{\norm{\pi_j(x)}_{V_j}^{1/j}\right \}.
\]
We will call this a {\em layered sup quasi-norm}.

To study horofunction boundaries, we will want to consider norms which induce genuine metrics, not quasi-metrics. Fortunately, a lemma of Guivarc'h tells us that as long as we are in a stratified group, we can upgrade this family of layered sup quasi-norms to genuine norms at the cost of introducing some constants.

\begin{lemma}(Guivarc'h, \cite{guivarch-quasinorm} Lemme II.1)\label{guivarch}
For any stratified group, up to rescaling each $\norm{\cdot}_{V_j}$ to a proportional norm $\lambda_j\,\norm{\cdot}_{V_j}$ if necessary, where $\lambda_1 = 1$, and  $0 < \lambda_{j+1}\leq \lambda_j$ for all $j$, the layered sup quasi-norm
\[
\norm{x} = \max_{1\leq j \leq s} \left\{\norm{\pi_j(x)}_{V_j}^{1/j}\right \}
\]
is, in fact, a norm satisfying the triangle inequality and hence induces a metric.
\end{lemma}

We will call these upgraded quasi-norms {\em layered sup norms}, and we will at times suppress the scalars $\lambda_j$, assuming that the ordinary norms $\norm{\cdot}_{V_j}$ have been scaled as necessary. The main focus of this paper will be to study the horofunction boundaries of these norms on families of Carnot groups which generalize the real Heisenberg group $\HR$.
In particular, we will consider {\em polysmooth} layered sup norms where each layer norm $\norm{\cdot}_{V_j}$ is either
\begin{enumerate}
\item {\em polyhedral}, i.e., the unit ball $Q_j$ of $\norm{\cdot}_{V_j}$ is a convex, centrally-symmetric polyhedron, or
\item {\em smooth}, meaning the norm is continuously differentiable on $V_j \setminus \{0\}$ and is strictly convex.
\end{enumerate}

We note that the homogeneous sup quasi-norm $|(x_1, \ldots, x_n)|_\infty = \max_{j}\left\{ |x_j|^{1/\nu_j}\right \}$ can be realized as a layered sup quasi-norm where we let each $\norm{\cdot}_{V_j}$ be the sup norm on that layer. Thus, we can apply Lemma~\ref{guivarch} and upgrade $|\cdot|_\infty$ to a homogeneous norm, which we will call the {\em homogeneous sup norm}.

%%%%%%%%%%%%%%%%%%%%%%%%%%%%%%%%%%%%%%%%%%%%%%%%%%%%%%%%%%%%%%%%%%%%%%%%%%%%%%%%%%%%%%%%%%%%

\subsection{Convex geometry}\label{sec:convex}

As we consider layered sup norms on stratified groups, we will want to make use of terminology and results coming from convex geometry. We refer the reader to \cite{beer-convex}, \cite{rockafellar-convex}, and \cite{schilling-thesis} for more details. 

For a subset $S$ of a vector space, we let $V(S)$ be the smallest vector subspace containing $S$. The affine hull $\aff(S)$ is the smallest affine subset containing $S$. We note that every affine subspace is the translation of a unique vector subspace. We call the unique subspace of which $\aff(S)$ is a translation the {\em space of translations} of $\aff(S)$, and denote it by $T(\aff(S))$.

For a subset $S$, we define the {\em relative interior} of $S$ to be the interior of $S$ within the affine subspace $\aff(S)$. Then we can define the {\em relative boundary} of $S$ as the closure of $S$ minus the relative interior of $S$.

We remind the reader that the unit ball of any normed linear space is a centrally-symmetric convex body with non-empty interior. A {\em supporting hyperplane} of a convex body $Q$ in $\R^k$ is a hyperplane satisfying two properties:
\begin{enumerate}
\item $Q$ is contained in one of the two closed half spaces defined by the hyperplane, and
\item $Q$ has at least one boundary point on the hyperplane.
\end{enumerate}

Throughout this paper we assume that each of the layer norms $\norm{\cdot}_{V_j}$ which define a layered sup norm on $G$ is either
\begin{enumerate}
\item {\em polyhedral}, i.e., the unit ball $Q_j$ of $\norm{\cdot}_{V_j}$ is a convex, centrally-symmetric polyhedron, or
\item {\em smooth}, i.e., the norm is continuously differentiable on $V_j \setminus \{0\}$ and $Q_j$ is strictly convex.
\end{enumerate}

For one of these convex unit balls $Q_j \subseteq \R^k$, a nonempty convex subset $F \subseteq Q_j$ is a {\em face} of $Q_j$ if for any $x \in F$ and every closed interval with $I \subseteq Q_j$ with $x$ in the relative interior of $I$, we have $I \subseteq F$. A face of codimension 1 is called a {\em facet}.

If $Q_j$ is polyhedral, then it is the intersection of finitely many halfspaces, and each face of $Q_j$ can be described as the intersection of one or more supporting hyperplanes.

In the smooth case, we restrict our attention to strictly convex norms so that each point in the boundary of the defining convex body $Q_j$ is an {\em extreme point}. That is, in the smooth case the supporting hyperplane at each point of the boundary of the unit ball $Q_j$ is unique and intersects $Q_j$ only at that point.

For the nonempty convex unit ball $Q_j$, the polar dual $Q_j^\circ$ is a subset of the dual space $(\R^k)^\circ$ defined by:
$Q_j^\circ := \{y \in (\R^k)^\circ \mid \langle y|x\rangle \leq 1\; \forall x \in Q_j\},$ 
where $\langle y | x\rangle$ is the evaluation of $y$ at $x$. 
If $F \subseteq Q_j$ is a proper face of $Q_j$, we define the {\em exposed dual} $F^\circ$ of $F$ to be the set
\[
F^\circ := \{y \in Q^\circ\mid \langle y \mid f \rangle = 1 \; \forall f \in F\}.
\]

For example, consider Figure~\ref{fig:convex}. The unit diamond $Q$ is an example of a convex, yet not strictly convex, body, and it is the unit ball of the taxi-cab metric on $\R^2$. If we let $S = \{v_1,v_2\}$, then
\[ \aff(S) = \{(x,y)\mid y= -x + 1\}, \quad\text{and}\quad T(\aff(S)) = \{(x,y)\mid y = -x\}.
\]
We see that the polar dual of the unit diamond is the unit square, and the exposed dual of $F = \{v_1\}$ is $F^\circ = \{1\} \times [-1,1]$.

\begin{figure}[h]
\centering
\begin{tikzpicture}

\draw (-2,0) -- (2,0);
\draw (0,-2) -- (0, 2);
\fill[gray, opacity=.5] (-1,0) -- (0,1) -- (1,0) -- (0,-1) -- cycle;
\draw[thick] (-1,0) -- (0,1) -- (1,0) -- (0,-1) -- cycle;
\fill[blue] (1,0) circle (0.07) node [below right] {$v_1 = (1,0)$};
\fill (0,1) circle (0.07) node [above left] {$v_2 = (0,1)$};
\node at (-1,-1) {$Q$};

\begin{scope}[xshift = 6cm]
\draw (-2,0) -- (2,0);
\draw (0,-2) -- (0, 2);

\fill[gray, opacity=.5] (-1,1) -- (1,1) -- (1,-1) -- (-1,-1) -- cycle;
\draw[thick] (-1,1) -- (1,1) -- (1,-1) -- (-1,-1) -- cycle;
\draw[line width=2.5, blue] (1,1) -- node [above right]{$v_1^\circ$} (1,-1);
\node at (-1.4,-1.4) {$Q^\circ$};

\end{scope}

\end{tikzpicture}
\caption{Example of convex body, polar dual, and exposed dual.}\label{fig:convex}
\end{figure}

In Figure~\ref{fig:smooth}, we have a smooth norm whose unit ball is an ellipse. Note that every point on the boundary if $Q$ is an extreme point. 
We see that the polar dual of the ellipse is also an ellipse, and the exposed dual of each face (i.e., of each extreme point) is a singleton.

\begin{figure}[h]
\centering
\begin{tikzpicture}

\draw (-2,0) -- (2,0);
\draw (0,-2) -- (0, 2);
\filldraw[black, thick, fill = gray!50] (0,0) ellipse (1.5 and .8);
\fill[blue] (3/2,0) circle (0.07) node [below right] {$v_1 = (\frac32,0)$};
\fill[red] (0,4/5) circle (0.07) node [above left] {$v_2 = (0,\frac45)$};
\node at (-1.2,-1.2) {$Q$};

\begin{scope}[xshift = 6cm]
\draw (-2,0) -- (2,0);
\draw (0,-2) -- (0, 2);
\filldraw[black, thick, fill = gray!50] (0,0) ellipse (2/3 and 5/4);
\fill[blue] (2/3,0) circle (0.07) node [below right] {$v_1^\circ = (\frac23,0)$};
\fill[red] (0,5/4) circle (0.07) node [above left] {$v_2^\circ = (0,\frac54)$};
\node at (-1.1,-1.1) {$Q^\circ$};

\end{scope}

\end{tikzpicture}
\caption{Example of smooth convex body with extreme points and their exposed duals.}\label{fig:smooth}
\end{figure}

%%%%%%%%%%%%%%%%%%%%%%%%%%%%%%%%%%%%%%%%%%%%%%%%%%%%%%%%%%%%%%%%%%%%%%%%%%%%%%%%%%%%%%%%%%%%

\subsection{Horofunction boundary of a metric space}\label{sec:horobdry}
Let $(X,d)$ be a metric space and $\cC(X)$ the space of continuous functions $X\to\R$ endowed with the topology of the uniform convergence on compact sets.
We fix a base point $x_0 \in X$ and define $\iota: X \to \cC(X)$ via 
\[
\iota: x \mapsto d(x, \cdot) - d(x, x_0).\] 

We define the \emph{horofunction compactification} $\overline{X}^h$ and the \emph{horofunction boundary} $\bhor(X)$
\[
\overline{X}^h = \overline{\iota(X)}, \text{  and  } \bhor X := \overline{\iota(X)} \setminus \iota(X).
\]
Under relatively mild conditions, the map $\iota$ is an embedding, i.e., a homeomorphism onto its image, making $\overline{X}^h$ a topological compactification. See \cite{kramer, FNG}.

Unlike other geometric boundary theories, the horofunction compactification, or metric compactification, can be used in very general settings without making assumptions on the geometry of the space. Additionally, the horofunction compactification is always Hausdorff, which, for example, is not necessarily true of the visual boundary in nilpotent groups and other spaces of mixed curvature \cite{FNG}. Finally, isometries of the metric space extend to homeomorphisms of the horofunction boundary, which has proved useful in the study of the dynamics of isometries and the asymptotic behavior of random walks \cite{karlsson-stars, karlsson-ledrappier, karlsson-ledrappier-drift, maher-tiozzo}.

\subsection{Pansu differentiability and horofunctions} \label{sec:Pansudifferentiability}
Let $G$ and $G'$ be two Carnot groups with homogeneous metrics $d$ and $d'$, respectively, and let $\Omega\subseteq G$ open.
A function $f:\Omega\to G'$ is \emph{Pansu differentiable at $p\in\Omega$} if 
there is homogeneous Lie group homomorphism $L$ 
such that
\[
\lim_{x\to p}
	\frac{d'( f(p)^{-1}f(x) , L(p^{-1}x) )}{d(p,x)}
= 0 .
\]
The map $L$ is called \emph{Pansu derivative} of $f$ at $p$ and it is denoted by $\pD f(p)$ or $\pD f|_p$.
A function $f:\Omega\to G'$ is \emph{strictly Pansu differentiable at $p\in\Omega$} if 
there is a homogeneous Lie group homomorphism $L$ 
such that
\[
\lim_{\epsilon\to0}\ \sup \left\{
	\frac{d'( f(y)^{-1}f(x) , L(y^{-1}x) )}{d(x,y)}
	: x,y\in B_d(p,\epsilon),\ x\neq y \right\}
= 0,
\]
where $B_d(p, \epsilon)$ is the open $\epsilon$-ball centered at $p$.
Clearly, in this case $f$ is Pansu differentiable at $p$ and $L=\pD F|_p$.

If we take $d$ to be a homogeneous metric on a graded group $G$ with identity element $e$, we can make use of Lemma 2.3 from \cite{FNG} which states that every horofunction can be realized as a generalized directional derivative of the metric function $d_e(x) = d(e, x)$ at a point on the unit metric sphere $\partial B$. That is,
	if $f\in\bhor(G,d)$, then there is a sequence of pairs $(p_n,\epsilon_n)\in\partial B \times(0,+\infty)$ such that $p_n\to p\in \partial B$, $\epsilon_n\to0$ and
	\begin{equation*}
	f(x) = \lim_{n\to\infty} \frac{d_e(p_n\delta_{\epsilon_n}x)-d_e(p_n)}{\epsilon_n}, 
	\quad\text{locally uniformly in $x\in G$.}
	\end{equation*}
%	locally uniformly in $x\in\HH$.
		
	On the other hand, if $(p_n,\epsilon_n)\in G\times(0,+\infty)$ such that $p_n\to p\in\partial B$, $\epsilon_n\to0$ and $f:G\to\R$ is the locally uniform limit $f(x) = \lim_{n\to\infty} \frac{d_e(p_n\delta_{\epsilon_n}x)-d_e(p_n)}{\epsilon_n}$, then $f\in\bhor(G,d)$.
	The horofunction $f$ is limit of the sequence of points
	\begin{equation*}
	\iota(q_n) = \iota(\delta_{1/\epsilon_n} p_n^{-1} ).
	\end{equation*}
	Moreover, if $d_e$ is strictly Pansu differentiable at $p$, then $f=\pD d_e|_p$;
	if $p_n\equiv p$ and $d_e$ is Pansu differentiable at $p$, then $f=\pD d_e|_p$.

This lemma tells us that every horofunction in the boundary of a homogeneous group looks like a directional derivative of the metric at a point on the unit metric sphere. In practice, by describing all possible types of limits at each point on the unit metric sphere, we can give a complete description of the horofunction boundary.

This method of studying the boundary also allows us to associate points on the unit sphere with horofunctions at infinity. Since Carnot groups come equipped with dilations, we can think of points on the unit sphere as a parametrization space for all dilation tracks and, therefore, as a parametrization space of all asymptotic directions. Thus we can give a correspondence between asymptotic directions and points in the horofunction boundary.

\subsection{Blow-ups and horofunctions}\label{sec:blow-ups}
As we observed above, in graded groups equipped with homogeneous distances, there is a connection between horofunctions in the boundary and generalized directional derivatives of the metric norm at points on the unit sphere $\partial B$. Almost everywhere on the unit sphere, the metric norm is Pansu differentiable, and these Pansu derivatives are horofunctions in the horofunction boundary. At points where the sphere is not Pansu differentiable, we need some more tools to study the generalized directional derivatives.

Here we will outline a few of the main definitions and results (without proofs) about {\em blow-ups} as defined in \cite{FNG}, which make use of the Kuratowski-Painlev\'e convergence of closed sets. This brief overview will allow us to state new results below, and we refer the reader to Section 3 of \cite{FNG} for a more detailed exposition.

\newcommand{\CL}{\mathtt{CL}}
Let $(X,d)$ be a locally compact metric space and let $\CL(X)$ be the family of all closed subsets of $X$.
If $x\in X$ and $C\subseteq X$, we set $d(x,C):=\inf\{d(x,y):y\in C\}$.
The {\em Kuratowski limit inferior} of a sequence $\{C_n\}_{n\in\N}\subset\CL(X)$ is defined to be
\begin{align*}
\Li_{n\to \infty} C_n 
&:= \left\{ q \in X : \limsup_{n\to \infty} d(q, C_n) = 0 \right\} \\
&= \left\{q\in X: \forall n\in\N\;\exists x_n\in C_n\text{ s.t.~}\lim_{n\to\infty}x_n=q \right\},
\end{align*}
while the {\em Kuratowski limit superior} is defined to be
\begin{align*}
\Ls_{n\to \infty} C_n 
&:= \left\{ q \in X : \liminf_{n\to \infty} d(q, C_n) = 0\right\} \\
&= \left\{q\in X: \exists N\subset\N\text{ infinite }\forall k\in N\;\exists x_{k}\in C_{k}\text{ s.t.~}\lim_{k\to\infty} x_{k}=q \right\}.
\end{align*}
From the definitions, one can see that $\Li_n C_n \subseteq \Ls_n C_n$ and that they are both closed.

If $\Li C_n = \Ls C_n = C$, then we say that the $C$ is the {\em Kuratowski limit} of $\{C_n\}_n$ and we write
\[
C=\Klim_{n\to\infty} C_n .
\]

Let $G$ be a homogeneous group with distance function $d$.
If $\Omega\subseteq G$ is closed, $\{p_n\}_{n\in\N}\subseteq G$ and $\{\epsilon_n\}_{n\in\N}\subset(0,+\infty)$ are sequences, we define the \emph{blow-up set}
\[
\BU(\Omega,\{p_n\}_n,\{\epsilon_n\}_n) 
:= \Klim_{n\to \infty} \delta_{1/\epsilon_n}(p_n\inv \Omega) ,
\]
if it exists. If the sequence $\{p_n\}_n$ is identically equal to the point $p$ and if the limit exists, we call this the {\em principal blow-up} at $p$.

\begin{prop}\label{prop03241537}{\cite[Proposition 3.5]{FNG}}
	Let $\Omega\subseteq G$ be a nonempty closed set. In the particular case that $p_n\equiv p$ and $\{\epsilon_n\}_{n\in\N}\subset(0,+\infty)$ with $\epsilon_n\to 0$, we have
	\begin{enumerate}
	\item
	If $p\notin\Omega$, then $\BU(\Omega,\{p_n\}_n,\{\epsilon_n\}_n)=\emptyset$.
	\item
	If $p\in \Omega^\circ$, the interior of $\Omega$, then $\BU(\Omega,\{p_n\}_n,\{\epsilon_n\}_n)= G$.
	\end{enumerate}
\end{prop}

One takeaway from this proposition is that the interesting examples occur when $p$ is in the boundary of $\Omega$. For example, let $\Omega$ be the closed unit square $\Omega = [-1,1]\times[-1,1] \subseteq \R^2$, and let $p = (1,0) \in \partial \Omega$. Then we have:
\begin{enumerate}
\item If $p_n \equiv p$ and $\epsilon_n \to 0$, then $\BU(\Omega,\{p\}_n,\{\epsilon_n\}_n) = \{(x,y) \mid x\leq 0\}$.
\item If $p_n = \left(1 + \frac cn,0\right)$ for $c \in \R$ and $\epsilon_n =\frac1n$, then $\BU(\Omega,\{p_n\}_n,\{\epsilon_n\}_n) = \{(x,y) \mid x\leq -c\}$.
\item If $p_n = \left(1 + \frac1n,0\right)$ and $\epsilon_n =\frac1{n^2}$, then $\BU(\Omega,\{p_n\}_n,\{\epsilon_n\}_n) = \emptyset$.
\end{enumerate}

Often times our closed regions $\Omega$ will be defined, at least locally, as the sublevel sets of smooth functions $F_j$. In this case, blow-ups are related to the Pansu derivatives of the region-defining functions $F_j$.

\begin{prop}\label{prop5e8c4f94}{\cite[Proposition 3.7]{FNG}}
	Let $\Omega\subseteq G$ be a nonempty closed set and $p\in\partial\Omega$.
	Suppose that there exists a neighborhood $U$ of $p$ and a finite family of smooth functions $F_j:U\to\R$ with $j\in J$ such that $\Omega\cap U = \bigcap_{j\in J}\{F_j\le0\}$ and $F_j(p)=0$.
	Suppose also the Pansu derivatives $\{\pD F_j|_p\}_{j\in J}$ are linearly independent.
	
	Then the principal blow-up is given by 
	\[ \BU(\Omega,\{p\}_n,\{\epsilon_n\}_n) = \{x\in G: \pD F_j|_p(x)\le 0,\ j\in J\}.   \]
	
	Also, for every $(t_j)_j\in(\R\cup \{+\infty\})^J$ and every $\epsilon_n\to 0^+$, there are $p_n\to p$  such that 
	\begin{equation}\label{eq5f75bda1}
	\BU(\Omega,\{p_n\}_n,\{\epsilon_n\}_n)
	= \{x\in G: \pD F_j|_p(x)\le t_j,\ j\in J\} .
	\end{equation}
\end{prop}

That is, the cone $\cap_{j\in J} \{\pD F_j\vert_p(x) \leq 0\}$ is the principal blow-up at $p$, and by perturbing the sequences $\{p_n\}_n$ and $\{\epsilon_n\}_n$, we can realize any finite or infinite translate of this cone as a blow-up at $p$.

For example, let $p = (1,-1) \in \Omega = [-1,1]\times[-1,1] \subseteq \R^2$. Then $\Omega$ is defined locally near $p$ by $x -1 \leq 0$ and $-y - 1 \leq 0$, and the principal blow-up at $p$ is
\[\BU(\Omega,\{p\}_n,\{\epsilon_n\}_n) = \{(x,y) \mid x\leq 0, y \geq 0\}.\]
Proposition~\ref{prop5e8c4f94} tells us that for any translate of this cone, for example, $\{x\leq 5, y \geq 1\}$ or $\{y \geq 4\}$, there are sequences $p_n \to p$ and $\epsilon_n \to 0$ whose blow-ups give that translated cone.

\subsubsection{Blow-ups of functions}

If, for all $n\in\N$, $\Omega_n\subseteq X$ are closed sets and $f_n:\Omega_n\to\R$ are continuous functions, then we say that, for some $\Omega\subseteq X$ closed and $f:\Omega\to\R$ continuous,
\[
\Klim_{n\to\infty}(\Omega_n,f_n) = (\Omega,f)
\]
if $\Omega=\Klim_n\Omega_n$ and if, for every $x\in\Omega$ and every sequence $\{x_n\}_{n\in\N}$ with $x_n\in\Omega_n$ and $x_n\to x$, we have $f(x)=\lim_nf_n(x_n)$.
Notice that this is equivalent to saying that 
\[
\Klim_{n\to\infty} \{(x,f_n(x)):x\in\Omega_n\} = \{(x,f(x)):x\in\Omega\} .
\]

For a continuous function $f:\Omega\to\R$, we define
\[
\BU((\Omega,f),\{p_n\}_n,\{\epsilon_n\}_n) 
:= \Klim_{n\to \infty} \left( \delta_{1/\epsilon_n}(p_n\inv \Omega) , \frac{ f(p_n\delta_{\epsilon_n}\cdot)-f(p_n) }{ \epsilon_n } \right) .
\]
As before, if $p_n \equiv p$ and if this limit exists, we call this the {\em principal blow-up} of the function $f$ on $\Omega$ at $p$.

\begin{prop}\label{prop5e8c979c}{\cite[Proposition 3.8]{FNG}}
	Let $\Omega\subseteq G$ be a nonempty closed set, $\{p_n\}_{n\in\N}\subseteq G$ and $\{\epsilon_n\}_{n\in\N}\subset(0,+\infty)$ sequences with $p_n\to p\in\Omega$ and $\epsilon_n\to 0$.
	Suppose that $\Omega_0:=\BU(\Omega,\{p_n\}_n,\{\epsilon_n\}_n)$ exists.
	Let $f: G\to\R$ be a continuous function that is strictly Pansu differentiable at $p$.
	Then
	\[
	\BU((\Omega,f),\{p_n\}_n,\{\epsilon_n\}_n) 
	= ( \Omega_0 , \pD f(p)|_{\Omega_0} ) .
	\]
\end{prop}

Looking forward, our horofunctions will be piecewise-linear functions. We will be able to partition our groups into closed dilation cones on which the norm $d_e(\cdot) = d(e, \cdot)$ is strictly Pansu differentiable. For each of these closed dilation cones $\Omega$, the blow-up $\BU((\Omega,d_e),\{p_n\}_n,\{\epsilon_n\}_n)$ will give the partial function $\pD d_e(p)$ and the subdomain $\Omega_0$ on which the partial function is defined. The next theorem tells us how to piece together these partial functions to get the blow-up functions which will be our horofunctions.

\begin{theorem}{\cite[Theorem 3.9]{FNG}}\label{piecetogether}
	Let $\Omega\subseteq G$ be a closed set such that there is a family $\mathcal Q$ of regular closed sets  with disjoint interiors such that $\Omega=\bigcup_{Q\in\mathcal Q}  Q$.
	For each $Q\in\mathcal Q$, let $f_Q: G\to\R$ smooth such that the function $f:\Omega\to\R$ defined by
	\[
	f(x) := \chi(x) \sum_{Q\in\mathcal Q} f_Q(x) \one_{ Q}(x) 
	\]
	is Lipschitz continuous, where $\chi(x) := \left( \sum_{Q\in\mathcal Q} \one_{ Q}(x) \right)^{-1}$.
	
	Let $\{p_n\}_{n\in\N}\subseteq G$ and $\{\epsilon_n\}_{n\in\N}\subset(0,+\infty)$ sequences with $p_n\to p\in\Omega^\circ$, the interior of $\Omega$, and $\epsilon_n\to 0$.
	Assume that $R_Q:=\BU(Q,\{p_n\}_n,\{\epsilon_n\}_n)$ exists for every $Q\in\mathcal Q$.

	Then 
	\[
	 G = \bigcup_{Q\in\mathcal Q} R_Q
	\]
	and $\BU((\Omega,f),\{p_n\}_n,\{\epsilon_n\}_n) = ( G,g)$ exists, where
	\begin{equation}\label{eq03271750}
	g(x) = \tilde\chi(x) \sum_{Q\in\mathcal Q} \left( \pD f_Q|_p(x) + c_Q \right) \one_{R_Q}(x),
	\end{equation}
	with $\tilde\chi(x) := \left( \sum_{Q\in\mathcal Q} \one_{R_Q}(x) \right)^{-1}$
	and $c_Q\in\R$.
\end{theorem}

%%%%%%%%%%%%%%%%%%%%%%%%%%%%%%%%%%%%%%%%%%%%%%%%%%%%%%%%%%%%%%%%%%%%%%%%%%%%%%%%%%%%%%%%%%%%
%%%%%%%%%%%%%%%%%%%%%%%%%%%%%%%%%%%%%%%%%%%%%%%%%%%%%%%%%%%%%%%%%%%%%%%%%%%%%%%%%%%%%%%%%%%%

\section{Horofunctions in general Carnot groups}\label{section-general}

Let $G$ be a stratified group equipped with a layered sup metric induced by the norm
\[
\norm{x} = \max_{1\leq j \leq s} \left\{\norm{\pi_j(x)}_{V_j}^{1/j}\right \},
\]
as defined in Section~\ref{subsec:quasi-norms}. 
Throughout this paper we assume that the layered sup norm is polysmooth. That is, we assume that each layer norm $\norm{\cdot}_{V_j}$ is either
\begin{enumerate}
\item {\em polyhedral}, i.e., the unit ball $Q_j$ of $\norm{\cdot}_{V_j}$ is a convex, centrally-symmetric polyhedron, or
\item {\em smooth}, meaning the norm is continuously differentiable on $V_j \setminus \{0\}$ and is strictly convex.
\end{enumerate}

One of our goals is to understand what types of functions can arise as horofunctions in the Carnot setting. In previous work with Sebastiano Nicolussi Golo, the author found that all horofunctions of the 3-dimensional real Heisenberg group $H(\R)$ when equipped with any sub-Finsler metric were piecewise-linear functions of coordinates coming from the first layer of the Heisenberg grading. We now consider all Carnot groups equipped with layered sup norms and show that an analogous result holds.

\begin{thma}
If $G$ is a stratified group equipped with a polysmooth layered sup norm, then all blow-ups of the norm at points $p$ on the unit sphere, and hence all horofunctions, are piecewise-linear and are functions of the coordinates of the first layer of the grading.
\label{thm:genPD}
\end{thma}

We will prove Theorem A by first computing the Pansu derivatives of the layered sup norms at certain points on the unit metric spheres. We will then be able to apply the results from \cite{FNG} as given in Section~\ref{sec:blow-ups}.

\begin{prop}\label{prop:pansuDiff} Suppose that $p$ is a point on the unit metric sphere $\partial B$ of a polysmooth layered sup norm on $G$ such that $\norm{\pi_i(p)}_{V_i}^{1/i} = 1$ and such that $\norm{\pi_j(p)}_{V_j}^{1/j} < 1$ for $j\neq i$. If $\pi_i(p)$ is in the interior of a facet of the polyhedron $Q_i$ or if $\norm{\cdot}_{V_i}$ is smooth, then the layered sup norm
\[
\norm{x} = \max_{1\leq j \leq s} \left\{\norm{\pi_j(x)}_{V_j}^{1/j}\right \}
\]
is Pansu differentiable at $p$. The principal blow-up or Pansu derivative is a linear function in the coordinates of $x$ coming from the first layer $V_1$ of the stratification.
\end{prop}

\begin{proof}
Suppose that the $i$th layer $V_i \cong \R^r$ and that $\pi_i(x) = (x_{i_1},\ldots, x_{i_r})$. To check the Pansu differentiability of the norm $\norm{\cdot}$ at the point $p$ as above, we compute
\begin{align*}
\lim_{t\to 0}\frac{\norm{p\delta_t x} - \norm{p}}{t}  &= \lim_{t\to0} \frac{\norm{\pi_i(p\delta_t x)}_{V_i}^{1/i}-1}{t}\\
&= \lim_{t\to0} \frac{1 /i}{\norm{\pi_i(p\delta_t x)}_{V_i}^{(i-1)/i}}\cdot \ddt \norm{\pi_i(p\delta_t x)}_{V_i}\\
&=  \frac{1}{i}\cdot \lim_{t\to0} \;\ddt \norm{((p\delta_tx)_{i_1},\ldots, (p\delta_tx)_{i_r})}_{V_i}\\
&=  \frac{1}{i}\cdot \lim_{t\to0} \; \sum_{k=1}^{r} \left.\frac{\partial}{\partial x_{i_k}}\right\vert_{\pi_i(p\delta_tx)}\norm{(x_{i_1},\ldots,x_{i_r})}_{V_i} \:\ddt (p\delta_tx)_{i_k}.
\end{align*}
By assumption, our layer norms $\norm{\cdot}_{V_j}$ are continuously differentiable in a small neighborhood containing $p$, and hence each partial derivative $\frac{\partial}{\partial x_{i_k}}\norm{\pi_i(p\delta_t x)}_{V_i}$ exists for $t$ sufficiently small. As $t \to 0$,  the limit $\lim_{t\to 0} \frac{\partial}{\partial x_{i_k}}\norm{\pi_i(p\delta_t x)}_{V_i} = \frac{\partial}{\partial x_{i_k}}\norm{\pi_i(p)}_{V_i}$, which does not depend on $x$.

Since $(p\delta_tx)_{i_k}$ is a polynomial in the coordinates of $p$ and $\delta_tx$, the limit $\lim_{t\to 0}\ddt (p\delta_tx)_{i_k}$ will pick out the coefficient of the linear term $t$ in $(p\delta_t x)_{i_k}$. From Proposition~\ref{prop:polynomials}, we know that the terms in the coordinate $ (p\delta_tx)_{i_k}$ of the product which are linear in $t$ are
\[
\sum_{ [\alpha]=i-1, [\beta]=1} c_{k,\alpha,\beta}(p)^\alpha(\delta_tx)^\beta.
\]
Hence the coefficient of $t$ in $(p\delta_t x)_{i_k}$ is a linear combination of products of the coordinates of $p$, which are fixed, with linear expressions in coordinates of $x$ coming from the first layer of the grading.

We conclude, therefore, that the polysmooth layered sup norm is Pansu differentiable at $p$, and the derivative is a linear function of the coordinates of $x$ coming from the first layer of the stratification.

\end{proof}

Note that when $G$ is a filiform group, since every layer except $V_1$ is 1-dimensional, any layered sup norm, modulo scaling layer norms by a constant $\lambda_j$ as in Guivarc'h's Lemma~\ref{guivarch}, is identical to the homogeneous sup norm $\norm{x} = \max_j \left\{ |x_j|^{1/\nu_j}\right \}$ except possibly in the first layer. Since we will study the horofunction boundaries of filiform groups in Section~\ref{sec:filiform}, it will be useful to look at the particular case of Proposition~\ref{prop:pansuDiff} for the homogeneous sup norm.

Let $p = (a_1, \ldots, a_r)$ and suppose that $|a_i| = 1$ while $|a_j| < 1$ for $j \neq i$. If we let $c_{i,k}$ denote the coefficient of $t^k$ in the $i$th coordinate fo the product $p\delta_tx$, then we have
\begin{align*}
\lim_{t\to0}\frac{\norm{p \delta_tx}-\norm{p}}{t} & = \lim_{t\to 0} \frac{\left| a_i + t^{\nu_i}x_i + \sum_{ [\alpha] + [\beta] = \nu_i} c_{i,\alpha,\beta}(p)^\alpha(\delta_tx)^\beta \right|^\frac1{\nu_i} - 1}{t}\\
& = \lim_{t\to 0} \frac{\left| x_it^{\nu_i} + c_{i,\nu_i-1}t^{(\nu_i-1)} + \cdots + c_{i,1}t + a_i \right|^\frac1{\nu_i} - 1}{t}\\
& = \pm \frac{c_{i,1}}{\nu_i|a_i|^{(\nu_i-1)/\nu_i}} = \pm \frac{c_{i,1}}{\nu_i},
\end{align*}

where the sign in the last line depends on the sign of $a_i$.

\begin{prop}\label{prop:edges} Suppose that $p \in \partial B$, that $\norm{\pi_i(p)}_{V_i}^{1/i} = 1$, and that $\norm{\pi_j(p)}_{V_j}^{1/j} < 1$ for $j\neq i$. Additionally assume that $\norm{\cdot}_{V_i}$ is polyhedral and that $\pi_i(p)$ belongs to a face of $Q_i$ of codimension greater than or equal to 2. In general, the polysmooth layered sup norm
is not Pansu differentiable at $p$, but its principal blow-up is a piecewise-linear function of the coordinates of $x$ coming from the first layer $V_1$ of the stratification.
\end{prop}

\begin{proof}
Following the same arguments as in the proof of Proposition~\ref{prop:pansuDiff}, we have

\begin{align*}
\lim_{t\to 0}\frac{\norm{p\delta_t x} - \norm{p}}{t}
&= \frac{1}{i}\cdot \lim_{t\to0} \;\ddt \norm{\pi_i(p\delta_t x)}_{V_i}.
\end{align*}

Since this norm is polygonal, locally near $\pi_i(p)$, the norm $\norm{\cdot}_{V_i}$ is piecewise-linear where each subdomain is a linear cone.
For a fixed $x \in G$, the curve $\{p\delta_tx\mid t\geq0\}$ is a smooth curve based at $p$ which is polynomial in $t$. Hence for $t$ sufficiently small, the projected curve $\pi_i(p\delta_t x)$ must stay in a single linear cone, say $C_\alpha$ where $\norm{y}_{V_1} = \alpha \cdot y$ for all $y \in C_\alpha$. In fact, we could write down conditions on the coordinates $x$ which would guarantee that $\pi_i(p\delta_tx) \in C_\alpha$ for $t$ sufficiently small. These conditions will depend on the coefficients of the terms which are linear in $t$ in $\pi_i(p\delta_tx)$, and, therefore, will depend only on coordinates of $x$ coming from the first layer of the grading. For such $x \in G$, 

\begin{align*}
\frac{1}{i}\cdot \lim_{t\to0} \;\ddt \norm{\pi_i(p\delta_t x)}_{V_i} &= \frac1i \lim_{t\to 0}\ddt\; \alpha \cdot \pi_i(p\delta_tx)\\
& = \frac1i \;\alpha\cdot \left(\lim_{t\to 0}\;\ddt\; \pi_i(p\delta_tx)\right) = \frac1i \; \alpha \cdot (c_{i_1,1}, \ldots, c_{i_r,1}),
\end{align*}
where again, $c_{i,k}$ is the coefficient of $t^k$ in the $i$th coordinate of $p\delta_t x$. For the same reasons as above is a linear function in the coordinates of the first layer of the stratification. Since this holds for each of the partial functions and linear cones defining the polygonal norm, our principal blow-up function is piecewise-linear and a function of the coordinates of the first layer of the stratification.

\end{proof}

\begin{proof}[Proof of Theorem A]
Let $p$ be a point on the unit sphere $\partial B$ of a polysmooth layered sup norm satisfying $\norm{\pi_i(p)}_{V_i}^{1/i} = 1$ and $\norm{\pi_j(p)}_{V_j}^{1/j} < 1$ for $j\neq i$. For any $p$, whether $\norm{\cdot}_{V_i}$ is smooth or polyhedral,
Propositions \ref{prop:pansuDiff} and \ref{prop:edges} tell us that the principal blow-up at $p$ is either a linear or piecewise-linear function of the coordinates of $x$ coming from the first layer $V_1$ of the stratification. Fisher--Nicolussi Golo's theorem stated above as Theorem~\ref{piecetogether} then tells us how to piece together these blow-up functions when $p$ lies on a face of higher codimension, where more of the projections $\norm{\pi_j(p)}_{V_j}^{1/j}$ equal 1. Since we are piecing together linear and piecewise-linear functions of coordinates coming from the first layer of the grading, the resulting function has these same properties.
Finally, once we have the principal blow-ups at all of these points on $\partial B$, we can use Proposition~\ref{prop5e8c4f94} to see that all other blow-up functions are simply translates of the principal blow-ups, proving the result.
\end{proof}

\section{Horofunction boundaries of higher Heisenberg groups}
Let $\HHR$ be the real Heisenberg group of dimension $2n+1$ and of step 2. In exponential coordinates, $\HHR \cong \R^{2n+1}$, which we will think of as $\R^n \times \R^n \times \R$, where
\[
(x,y,z) (x', y', z') = \left(x + x', y + y', z + z' + \frac12(x\cdot y' - x'\cdot y)\right),
\]
and where $\cdot$ above represents the standard inner product on $\R^n$. We will equip $\HHR$ with a polysmooth layered sup norm
\[
\norm{(x,y,z)} = \max\left\{\norm{(x,y)}_{V_1}, \lambda\sqrt{|z|}\right\},
\]
where we will endow the first layer $V_1 \cong R^{2n}$ with a norm $\norm{\cdot}_{V_1}$ which is either
\begin{enumerate}
\item polyhedral, or
\item smooth, as defined in Section~\ref{section-general}.
\end{enumerate}
For generic metrics of this type, we are able to give explicit analytic and topological descriptions of the horofunction boundary. There are certain layered sup norms on $\HHR$ which exhibit special symmetries and for which typically disjoint parts of the horofunction boundary have nonempty intersection. We will say that these norms have {\em non-separated} boundaries. In these examples, our discussions below describe the horofunctions in the boundary, but the topological type of the boundary may differ from the general case. See Section~\ref{sec:degenerate} for a more precise definition of non-separation and Section~\ref{sec:deg_example} for an example of the boundary of a norm with non-separated boundary.

\begin{thmb}\label{thm:thmb}
The horofunction boundary of $\HHR$ equipped with a polysmooth layered sup norm has dimension $2n$. Except in non-separated cases, the horofunction boundary is homeomorphic to a $2n$-dimensional button pillow. That is, the boundary is homeomorphic to the sphere $S^{2n}$ with closed neighborhoods of the north and south poles identified, as in Figure~\ref{fig:button}.
\end{thmb}

To prove this theorem, we make use of the result \cite{FNG} described above in Section~\ref{sec:Pansudifferentiability}. Indeed, by computing all possible blow-ups, that is, generalized directional derivatives, of the homogeneous norm at points on the unit metric sphere, we will find all horofunctions in boundary. Recall that for a point $p$ on the unit sphere $\partial B$ and $q \in \HHR$, the principal blow-up $f$ at $p$ is
\[
f(q) = \lim_{t\to0} \frac{\norm{p\delta_tq} - \norm{p}}{t}.
\]
If this blow-up function is linear, the principal blow-up is the only blow-up of the norm at $p$. Otherwise, by Theorem~A the principal blow-up function is piecewise-linear and defined by at least two partial functions. In this case, the set of all blow-ups at this point is equal to the set of translates of the principal blow-up, as described in Section~\ref{sec:blow-ups}. We will break up the metric spheres into various regions to analyze the sets of blow-up functions. Identifying $\HHR$ with $\R^n \times \R^n \times \R$, we will think of $\R^n \times \R^n$ as horizontal directions and of $\R$ as the vertical component. In doing so, we break up $\partial B$ into interior ceiling and floor points, interior wall points, and seam points at the intersection of the walls with the ceiling or floor. See, for example, Figure~\ref{fig:button}.

\begin{figure}[ht]
\centering
\begin{tikzpicture}[scale=.75]

\node at (0,-.2) {\includegraphics[width=1.5in]{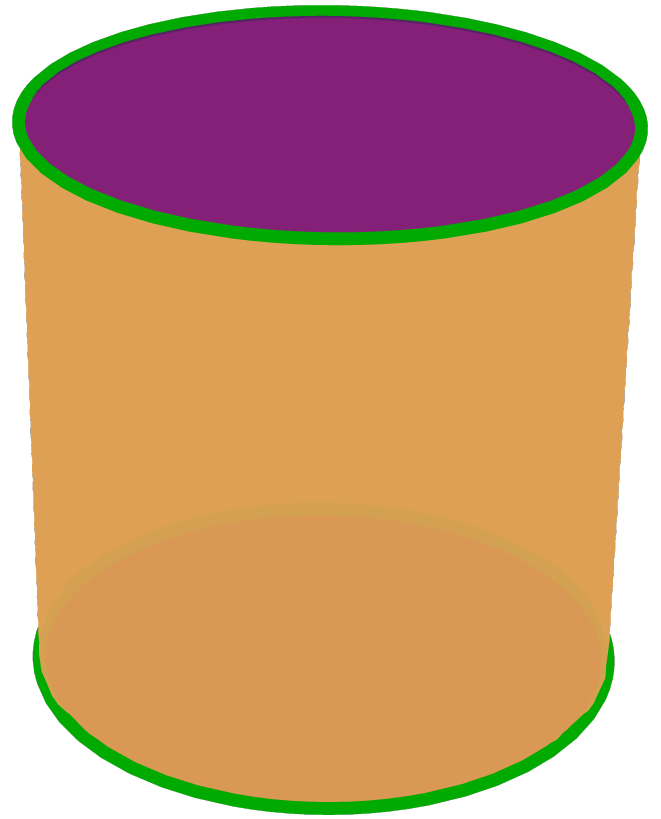}};
\node at (-3.5 ,1.5) [fill = red!50!blue, opacity =.7] {ceiling};
\node at (-3.5 ,-.5) [fill = orange, opacity =.7] {wall};
\node at (3.2 ,-.2) [fill = green!80!brown, opacity =.7] {seams};
\draw[very thick, ->] (-2.6, 1.5) to[bend left = 40] (-.5,2);
\draw[very thick, ->] (-2.8, -.6) to[bend right = 40] (-1,-1);
\draw[very thick, ->] (3.2, .2) to[bend right = 30] (2.5,1.8);
\draw[very thick, ->] (3.2, -.6) to[bend left = 30] (2.2,-2);

\node at (0,4) [fill = white, draw = black] {Unit metric sphere};
\draw[very thick, ->] (5.2, 0) to (6.2,0);

\begin{scope}[xshift = 11.5 cm, yshift = 0cm]

\node at (0,-.2) {\includegraphics[width=2.3in]{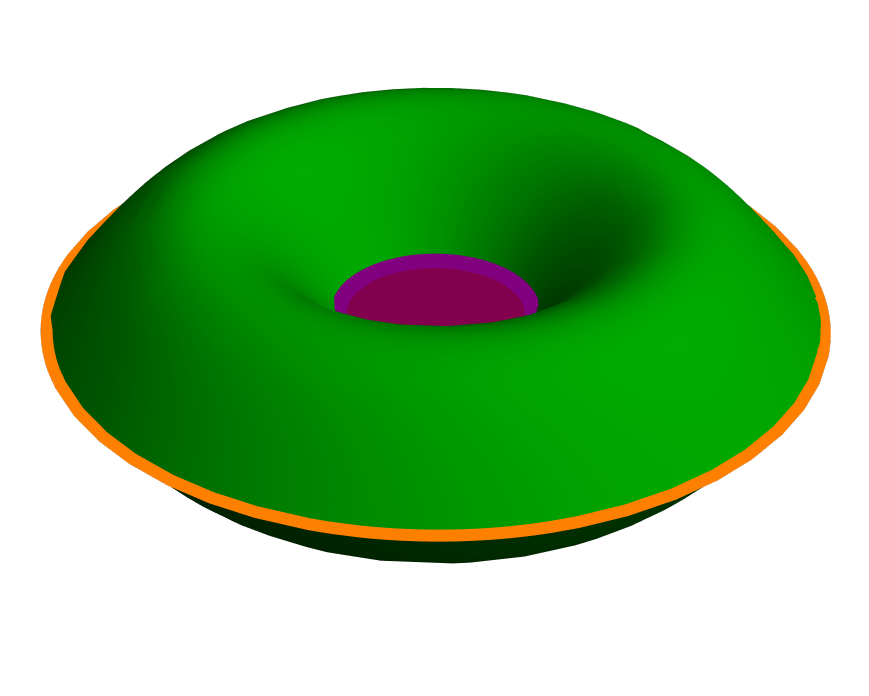}};

\node at (0,4) [fill = white, draw = black] {Horofunction boundary};

\end{scope}

\begin{scope}[yshift = -9cm, scale=1]

\draw[dashed] (-3,-1) -- (1, -1) -- (2.8, .8) -- (-1.2, .8) -- cycle;
\draw[dashed] (-1.2,.8) -- (-1.2, 4.6);

\fill[blue, opacity = .7] (-3, 2.8) -- (1, 2.8) -- (1, -1) -- (-3, -1) -- cycle;
\fill[blue, opacity = .7] (1, 2.8) -- (1, -1) -- (2.8, .8) -- (2.8, 4.6) --  cycle;
\fill[violet!90!black, opacity = .9] (-3, 2.8) -- (1, 2.8) -- (2.8, 4.6) -- (-1.2, 4.6) -- cycle;

\draw[green!70!black, line width = 1.5] (2.8, .8) -- (1, -1) -- (-3,-1);
\draw[green!70!black, line width = 1.5] (2.8, 4.6) -- (1, 2.8);
\draw[green!70!black, line width = 1.5] (-1.2, 4.6) -- (-3, 2.8);
\draw [orange, line width = 1.5] (-3, -1) -- (-3, 2.8);
\draw [orange, line width = 1.5] (1, 2.8) -- (1,-1);
\draw [orange, line width = 1.5] (2.8, .8) -- (2.8, 4.6);
\draw[green!70!black, line width = 1.5] (2.8, .8) -- (1, -1) -- (-3,-1);
\draw[green!70!black, line width = 1.5] (2.8, 4.6) -- (1, 2.8) -- (-3, 2.8) -- (-1.2, 4.6) -- cycle;

\node at (-4,3.3) [fill = red!50!blue, opacity =.7] {ceiling};
\node at (-4.2 ,1) [draw = orange, line width=1, fill = blue!70!white, opacity =.7] {walls};
\node at (3.8 ,1.7) [fill = green!80!brown, opacity =.7] {seams};
\draw[very thick, ->] (-3.1, 3.3) to[bend left = 40] (-.5,3.8);
\draw[very thick, ->] (-3.5, .6) to[bend right = 40] (-1.1,.4);
\draw[very thick, ->] (3.8, 2.1) to[bend right = 30] (2.2,3.9);
\draw[very thick, ->] (3.8, 1.3) to[bend left = 30] (2.2,0);

\draw[very thick, ->] (5.2, 2) to (6.2,2);

\end{scope}

\begin{scope}[xshift = 11.5 cm, yshift = -7cm]

\node at (0,-.2) {\includegraphics[width=2.3in]{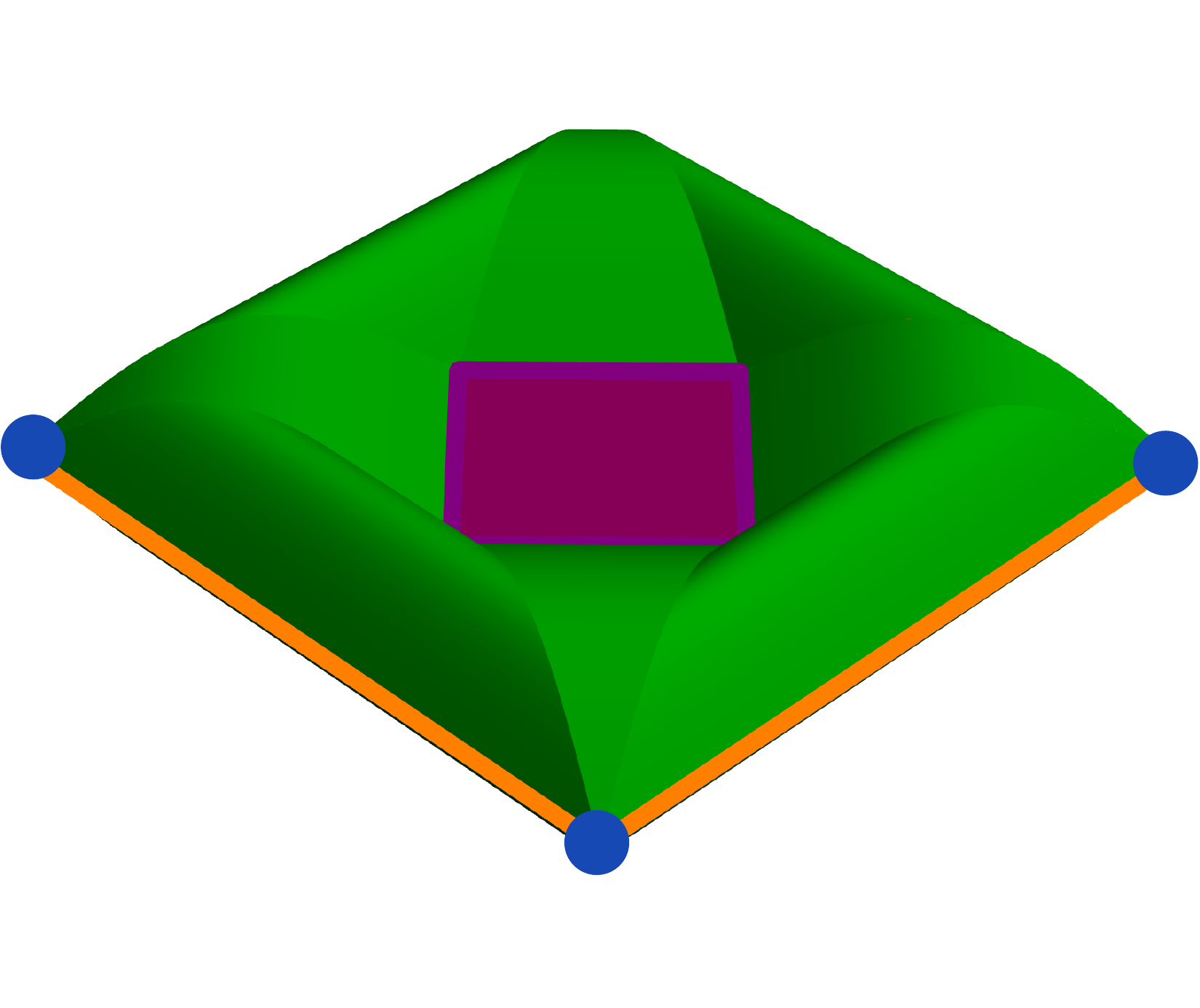}};

\end{scope}

\end{tikzpicture}
\caption{On the left we have the unit metric spheres in $\HR$ equipped with a Euclidean layered sup norm (above) and the homogeneous sup norm (below). On the right, we have the corresponding horofunction boundaries. Colors give correspondence between points on the sphere and blow-ups of the norm at those points.}\label{fig:button}
\end{figure}

\subsection{Blow-ups at interior ceiling/floor points}
First we consider points in the relative interior of the ceiling and floor of the unit sphere, i.e., points $p = (a,b,c) \in \partial B$ for which $\norm{(a,b)}_{V_1}<1$ and $\lambda \sqrt{|c|} = 1$.

\begin{lemma}{(Blow-ups at ceiling/floor points)}\label{lem:ceilingBU}
Let $p = (a, b, c)$ be a point in the interior of the ceiling or floor of $\partial B$, that is, such that $\norm{(a,b)}_{V_1} <1$ and $c = \pm 1/\lambda^2$. If $p$ is on the ceiling, then the blow-up of the norm at $p$ is given by
\[\phi^+_{(a,b)}(x, y, z) = \frac{\lambda^2}4 (-b\cdot x + a\cdot y).
\]
If $p$ is on the floor, then the blow-up of the norm at $p$ is given by
\[\phi^-_{(a,b)}(x, y, z) = \frac{\lambda^2}4 (b\cdot x - a\cdot y).\]
The set of all blow-up functions at interior ceiling is equal to the set of blow-up functions at interior floor points, and is homeomorphic to an open $2n$-dimensional ball.
\end{lemma}

\begin{proof}
Let $p = (a,b,c)$ be a point in the interior of the ceiling or floor of $\partial B$ as in the statement of the lemma. To compute the principal blow-up at $p$, we let $q = (x,y,z)$ and evaluate
\begin{align*}
\lim_{t\to0} \frac{\norm{p\delta_tq} - \norm{p}}{t} & = \lim_{t\to0}\frac{\max\left\{\norm{(a + tx, b + ty)}_{V_1}, \lambda\sqrt{|\pm \frac{1}{\lambda^2} + t^2z + \frac12 t(a\cdot y - b\cdot x)|}\right\} - 1}{t}\\
& = \lim_{t\to0}\frac{\lambda\sqrt{|\pm \frac{1}{\lambda^2} + t^2z + \frac12 t(a\cdot y - b\cdot x)|}- 1}{t}\\
& = \lim_{t\to0}\frac{\lambda\sqrt{\frac{1}{\lambda^2} \pm t^2z \pm \frac12 t(a\cdot y - b\cdot x)}- 1}{t}  = \pm \frac{\lambda^2}{4}(-b\cdot x + a\cdot y).
\end{align*}

As the principal blow-up function at $p$ is linear, this is the unique blow-up of the norm at $p$.
Since we are assuming that the norm $\norm{\cdot}_{V_1}$ is symmetric, we have that the set of interior ceiling point blow-ups $\{\frac{\lambda^2}{4}(-b\cdot x + a\cdot y)\mid \norm{(a,b)}_{V_1} < 1\}$ is equal to the set of interior floor point blow-ups $\{\frac{\lambda^2}{4}(b\cdot x - a\cdot y)\mid \norm{(a,b)}_{V_1} < 1\}$, and this set of functions is homeomorphic to an open $2n$-dimensional ball.
\end{proof}

We note that since the horofunction boundary is a closed set, the boundary points of this open $2n$-dimensional ball, that is, the functions $\{\frac{\lambda^2}4(-b\cdot x + a\cdot y)\mid \norm{(a,b)}_{V_1} =1\}$
also appear in the horofunction boundary. Indeed, these functions will appear as blow-ups of the norm at seam points of the unit sphere $\partial B$, as described in Section~\ref{sec:seam}.

\subsection{Blow-ups at interior wall points}

The blow-up functions at interior wall points, i.e., points $p=(a,b,c) \in \partial B$ for which $\norm{(a,b)}_{V_1} =1$ and $\lambda\sqrt{|c|} < 1$, will depend on the geometry of the polyhedral or smooth normed space $(\R^{2n}, \norm{\cdot}_{V_1})$. Indeed, as we consider a polysmooth layered sup norm on $\HHR$, our results on the blow-ups of the norm at interior wall points will rely upon the following result of Schilling.

\begin{theorem}[\cite{schilling-thesis}, Theorem 3.3.10]
The horofunction boundary of $(\R^n, \norm{\cdot})$, where $\norm{\cdot}$ is either polyhedral or smooth, is homeomorphic to $S^{n-1}$.
\end{theorem}

\begin{cor}{(Blow-ups at wall points)}
The collection of all blow-up functions at interior wall points, that is, at points $p = (a,b,c)$ such that $\norm{(a,b)}_{V_1} = 1$ and $\lambda\sqrt{|c|}<1$, is isomorphic to the boundary $\bhor(\R^{2n}, \norm{\cdot}_{V_1})$, which is homeomorphic $S^{2n-1}$.
\end{cor}

\begin{proof}
Let $p = (a,b,c)$ be an interior wall point of $\partial B$ as in the statement of the corollary. To compute the principal blow-up at $p$, we let $q = (x,y,z)$ and evaluate
\begin{align*}
\lim_{t\to0} \frac{\norm{p\delta_tq} - \norm{p}}{t} & = \lim_{t\to0}\frac{\max\left\{\norm{(a + tx, b + ty)}_{V_1}, \lambda\sqrt{|\pm \frac{1}{\lambda^2} + t^2z + \frac12 t(a\cdot y - b\cdot x)|}\right\} - 1}{t}\\
& =  \lim_{t\to0}\frac{\norm{(a + tx, b + ty)}_{V_1} - 1}{t}.
\end{align*}
We note that this is equal to the principal blow-up of $\norm{\cdot}_{V_1}$ at the point $(a,b)\in \R^{2n}$. Our horofunctions are independent of the vertical coordinate $z$, and the result follows from \cite{schilling-thesis}.
\end{proof}

To describe these horofunctions in more detail, we will make use of ideas from convex geometry as defined in Section~\ref{sec:convex}. We remind the reader that the affine hull $\aff(S)$ is the smallest affine subset containing $S$, and the unique vector subspace of which $\aff(S)$ is a translation is denoted by $T(\aff(S))$, or simply $T(S)$.

Let $Q_1$ denote the unit ball of $(\R^{2n},\norm{\cdot}_{V_1})$. For both polyhedral and smooth norms, in \cite{schilling-thesis} Schilling gives the horofunction boundary as
\[\bhor(\R^{2n},\norm{\cdot}_{V_1}) = \{h_{E,u}\mid E\subseteq Q_1^\circ \text{ is a proper face, } u \in T(E)^*\},\]
where $h_{E,u}: X \to \R, \quad y \mapsto |u - y|_E - |u|_E$,
and for a convex set $C \subseteq X^*$, we have  $|x|_C := -\inf_{q \in C} \langle q | x\rangle$. Schilling then goes on to describe a homeomorphism between this set and the boundary of the polar dual of the unit ball $Q_1$.

We now reframe these results in terms of blow-ups and Pansu derivatives. If $d_e = \norm{\cdot}_{V_1}$ is smooth, then the horofunction boundary of $(\R^{2n},\norm{\cdot}_{V_1})$ is given by
\[\Delta = \{\pD d_e\vert_{(a,b)}(x,y) = h_{(a,b)^\circ}(x,y) = \alpha\cdot x + \beta \cdot y \mid (\alpha, \beta) = (a,b)^\circ, (a,b) \in \partial Q_1\},
\]
where $(a,b)^\circ$ denotes the exposed dual of the extreme point $(a,b)$, as defined in Section~\ref{sec:convex}.

When $\norm{\cdot}_{V_1}$ is polygonal, explicitly writing down all of the horofunctions takes a bit more work. Using Schilling's notation, if we let $\{F_i\}_i$ be the finite set of faces of $Q_1$ of any dimension to which $(a,b)$ belongs, then the set of all blow-ups of the norm at $p= (a,b)$ is given by
\[\mathcal W_{(a,b)} = \bigcup_{F_i \ni (a,b)}\{h_{F_i^\circ, u}\mid u \in V(F_i)^\perp \}.
\]

We can also describe the horofunction boundary in another way. As stated above, the horofunction boundary is homeomorphic to $\partial Q_1^\circ$, which is itself a polyhedron. This dual polyhedron can be described by the union of finitely many $(2n-1)$-dimensional faces, each of which is the exposed dual of a vertex of $Q_1$. In the horofunction boundary, each of these $(2n-1)$-dimensional faces is the set of all blow-ups, that is, the set of both finite and infinite translates of the principal blow-up, of $d_e = \norm{\cdot}_{V_1}$ at a vertex of $Q_1$. 
The principal blow-up at a vertex will be a piecewise-linear function defined in terms of the codimension-1 facets which intersect at the vertex. Indeed, each facet $F_i$  of this polyhedron is defined by a linear equation $\alpha \cdot x + \beta \cdot y = 1$. The horofunctions in $\bhor (\R^n, \norm{\cdot}_{V_1})$ are linear and piecewise-linear functions built out of the finitely-many facet blow-up functions $\{h_{F_i^\circ}(x,y) = \alpha \cdot x + \beta \cdot y\mid F_i \text{ is a facet of } Q_1\}.$ 

\subsection{Disjointness of ceiling/floor and wall blow-up functions}\label{sec:degenerate}

We'd now like to examine if the two families of horofunctions we have found so far are disjoint or not. As the closure of blow-ups at interior ceiling and floor points, we have 
\[\Phi = \left\{\frac{\lambda^2}4(-b\cdot x + a\cdot y)\mid \norm{(a,b)}_{V_1} \leq 1\right\} = \overline{\{\phi^\pm_{(a,b)}\mid \norm{(a,b)}_{V_1} < 1\}}.
\]  As the blow-ups at interior wall points, we get $\Delta$, a copy of $\bhor(\R^{2n}, \norm{\cdot}_{V_1})$. (As discussed in the previous section, it is not as simple to give an explicit description of the functions in $\Delta$.) In the general case, which we will call the {\em separated} case, these two families are disjoint, and the blow-up functions coming from seam points will join the families to give us the boundary described in Theorem~B. There are examples, however, where these families are not disjoint, such as the one described in Section~\ref{sec:deg_example}. If we fix our norm $\norm{\cdot}_{V_1}$ on the first layer of $\HHR$, we can always find a open interval's worth of values positive scalars $\lambda$ such that the layered sup quasi-norm $\max \{\norm{(x,y)}_{V_1}, \lambda \sqrt{|z|}\}$ induces a genuine metric and whose boundary is separated.

\begin{lemma}
For every norm $\norm{\cdot}_{V_1}$ on $\R^{2n}$, there exists a positive $\lambda_0 < 1$ such that for all positive $\lambda \leq \lambda_0$, the layered sup quasi-norm $|(x,y,z)| = \max \{\norm{(x,y)}_{V_1}, \lambda \sqrt{|z|}\}$ on $\HHR$ is a genuine norm. Additionally, there exists a positive $\lambda_1 \leq \lambda_0$ such that for all positive $\lambda \leq \lambda_1$, the horofunction boundary of $\HHR$ equipped with the layered sup norm is separated.
\end{lemma}

\begin{proof}
For the first statement, Guivarc'h's result cited above as Lemma~\ref{guivarch} tells us that there exists a positive $\lambda_0$ such that $\max \{\norm{(x,y)}_{V_1}, \lambda_0 \sqrt{|z|}\}$ is a genuine norm on $\HHR$. The proof of Guivarc'h's lemma makes clear that any positive $\lambda\leq \lambda_0$ also suffices.

For the second statement, let $\Phi(\lambda) = \left\{\frac{\lambda^2}4(-b\cdot x + a\cdot y)\mid \norm{(a,b)}_{V_1} \leq 1\right\}$, and let $\Delta =\bhor(\R^{2n}, \norm{\cdot}_{V_1})$. We would like to find a $\lambda_1$ such that for all positive $\lambda < \lambda_1$, we have
$\Phi(\lambda) \cap \Delta = \emptyset$. Note that the family $\Phi$ is made up entirely of linear functions, and in fact it is a scaled and rotated copy of $Q_1$ in $(\R^{2n})^*$. On the other hand, $\Delta$ is independent of the scalar $\lambda$ and will contain nonlinear functions if $\norm{\cdot}_{V_1}$ is polyhedral. If $\norm{\cdot}_{V_1}$ is polyhedral, then there is a finite number of linear functions in $\Delta$, one for each of the codimension-1 faces of the polyhedron $Q_1$. If $\norm{\cdot}_{V_1}$ is smooth, then all functions in $\Delta$ are linear. Either way, the set of linear functions $\Delta \cap (\R^{2n})^*$ is a closed set which does not contain the origin. Hence there exists an open neighborhood $U$ of $0 \in (\R^{2n})^*$ which is disjoint from $\Delta \cap (\R^{2n})^*$. Let $\lambda_1'>0$ be a scalar such that $\Phi(\lambda_1') \subseteq U$. Then if $\lambda_1 = \min\{\lambda_0, \lambda_1'\}$, then $\Phi(\lambda) \cap \Delta = \emptyset$ for all $\lambda \leq \lambda_1$.
\end{proof}

\subsection{Blow-ups at seam points}\label{sec:seam}

Finally, we must examine the blow-ups of the polysmooth layered sup norm at points on the seams between the wall regions and the ceiling or floor regions of the unit sphere $\partial B$. So far in the separated case, we have a $2n$-dimensional ball of horofunctions, i.e., blow-ups, corresponding to interior ceiling or floor points and a $(2n-1)$-dimensional sphere of horofunctions corresponding to interior wall points. In the right panel of Figure~\ref{fig:button}, these are given by the purple disk and the orange circle, respectively.

Since a seam point $p$ is on the border between two regions on which the layered sup norm is measured in different ways, the principal blow-up function of the norm at $p$ will be a piecewise-defined function with partial functions corresponding to wall point blow-ups and ceiling point blow-ups.

\begin{lemma}{(Blow-ups at ceiling seam points)}
The set of blow-up functions at ceiling seam points, that is, at points $p = (a,b,c)$ such that $\norm{(a,b)}_{V_1} = 1$ and $c = \frac1{\lambda^2}$, is homeomorphic to a closed spherical shell $S^{2n-1}\times I$, for $I$ a closed interval. One boundary sphere of the spherical shell is the boundary sphere of $\Phi=\{\frac{\lambda^2}4(-b\cdot x + a\cdot y)\mid \norm{(a,b)}_{V_1} \leq 1\}$, and the other boundary sphere is $\Delta = \bhor(\R^{2n},\norm{\cdot}_{V_1})$.
\end{lemma}

\begin{proof}
Suppose $p = (a,b,c)$ is a ceiling seam point with $\norm{(a,b)}_{V_1} = 1$ and $c = \frac1{\lambda^2}$. Whether $\norm{\cdot}_{V_1}$ is smooth or polyhedral, the principal blow-up of the layered sup norm at $p$ will have the ceiling function $\phi^+_{(a,b)}$, as defined in Lemma~\ref{lem:ceilingBU}, as a partial function. Any two piecewise-defined functions can only be equal if they are defined by the same partial functions. Since $\phi^+_{(a,b)} = \phi^+_{(a',b')}$ if and only if $(a,b) = (a',b')$ and since we are assuming the boundary is separated, meaning no other partial functions in the blow-up are equal to $\phi^+_{(a,b)}$, the principal blow-up functions at points on the ceiling seam are pairwise distinct. Via the same argument, any two finite translates of the principal blow-ups at different ceiling seam points must also be distinct.

If $\norm{\cdot}_{V_1}$ is smooth, then the principal blow-up of the homogeneous norm at $p$ is a piecewise-linear function with partial functions $\phi^+_{(a,b)}$ and $h_{(a,b)^\circ}$, where the subdomains of these two partial functions are uniquely determined by requiring that the blow-up function is continuous. Indeed, these linear functions are distinct since
$\phi^+_{(a,b)}(a,b) = \frac{\lambda^2}4(-b\cdot a + a \cdot b) = 0$, while by the definition of the exposed dual, $h_{(a,b)^\circ}(a,b) = 1.$
Since there are only two partial functions in this principal blow-up function, there is a 1-parameter family of translates of this function, with limiting infinite translates given by $\phi^+_{(a,b)}$ and $h_{(a,b)^\circ}$. Hence we have a one-to-one correspondence between the boundary ceiling blow-up functions $\{\phi^+_{(a,b)}\mid \norm{(a,b)}_{V_1} = 1\}$ and the wall blow-up functions $\{h_{(a,b)\circ}\mid \norm{(a,b)}_{V_1} = 1\}$, and we get a 1-parameter family of blow-ups connecting each pair. Since $\phi^+_{(a,b)}$ and $h_{(a,b)\circ}$ vary continuously with $a$ and $b$, this gives us a set of blow-up functions which is homeomorphic to $S^{2n-1}\times I$.

When $Q_1$ is polyhedral, the construction of the spherical shell is more complex. Let $F$ is the face of $Q_1$ of minimal dimension to which $(a,b)$ belongs with $\dim F = k$. We know from above that the family of wall blow-up functions $\mathcal W_{(a,b)}$ associated to $(a,b)$ is $(2n-1-k)$-dimensional. We want show that in the principal blow-up at our ceiling seam point $p = (a,b,c)$, the subdomain of the partial function corresponding to the ceiling dilation cone is transverse to the subdomains of the partial functions corresponding to the wall dilation cone. Knowing this tells us that the ceiling seam principal blow-up corresponding to $(a,b)$ has one more dimension of translates than the corresponding principal wall blow-up.

Indeed, consider the norm $\norm{\cdot}_{V_1}$ locally near the point $(a,b) \in \R^{2n}$. It is piecewise-linear with subdomains which are linear cones based at $(a,b)$. That is, the boundaries of these subdomains are half-hyperplanes through the origin which all intersect at $(a,b)$ and hence along the line connecting the origin to $(a,b)$. The principal blow-up of $\norm{\cdot}_{V_1}$ at $(a,b)$, therefore, is invariant under translates in the direction of $(a,b)$.

The point $(a,b)$ is in intersection of finitely many facets $\{F_j\}_j,$ each defined by the unique hyperplane $H_j = \{(x,y)\mid \alpha_j \cdot x + \beta_j\cdot y = 1\}$. If we let $G_j (x,y,z) = (\alpha_j\cdot x + \beta_j\cdot y)^2 - \lambda^2z$, then the ceiling dilation cone, i.e., the set of points in the dilation cone of the closed ceiling face of $\partial B$, is defined locally near $p$ by the locus $ \cap_j\{G_j \leq 0\}$. By Proposition~\ref{prop5e8c4f94} and Theorem~\ref{piecetogether}, the principal blow-up of the homogeneous norm at $p$, then, has the partial function $\phi^+_{(a,b)}$ defined on the subdomain $\cap_j \{q \in \HHR \mid \pD G_j\vert_p(q) \leq 0\}$. Computing these Pansu derivatives, we see
\begin{align*}
\pD G_j\vert_p(x,y,z) &= \lim_{t \to 0} \frac{(\alpha_j (a + tx) + \beta_j\cdot (b + ty))^2 - \lambda^2(c + t^2z + \frac12t (a\cdot y - b\cdot x))}{t}\\
& = \left(2\alpha_j + \frac{\lambda^2}2 b\right)\cdot x + \left(2\beta_j - \frac{\lambda^2}2 a\right) \cdot y,
\end{align*}
and evaluated at $-p$, we see $\pD G_j\vert_p(-a, -b, -c) = -2$. Hence $-p$ is in the interior of the subdomain for the partial function $\phi^+_{(a,b)}$. We have, then, the boundaries of the subdomain for this ceiling blow-up function are transverse to the line connecting the origin to $p$, and we, therefore, have a new direction in which to take translates of the principal blow-up.

Again, we were assuming $F$ was the face of $Q_1$ of minimal dimension to which $(a,b)$ belongs with $\dim F = k$ and that the family of wall blow-up functions $\mathcal W_{(a,b)}$ associated to $(a,b)$ was $(2n-1-k)$-dimensional. We can conclude, therefore, that the set of blow-up functions corresponding to the ceiling seam point $p$ is $(2n-k)$-dimensional, which we can think of as the join of $\phi^+_{(a,b)}$ and the family $\mathcal W_{(a,b)}$. As we let $(a,b)$ vary along the $k$-dimensional face $F$, we get a $2n$-dimensional family of blow-ups, described as the join of $\{\phi^+_{(a,b)} \mid (a,b) \in F\}$ and $\mathcal W_{(a,b)}$ (which is equal to $\mathcal W_{(a',b')}$ if $(a,b)$ and $(a',b')$ belong to all the same faces). These $2n$-dimensional joins come together to create a $2n$-dimensional spherical shell.

\end{proof}

\begin{lemma}{(Blow-ups at floor seam points)}
The set of blow-up functions at floor seam points, that is, at points $p = (a,b,c)$ such that $\norm{(a,b)}_{V_1} = 1$ and $c = -\frac1{\lambda^2}$, also is homeomorphic to a closed $2n$-dimensional spherical shell with the same boundary spheres as the spherical shell of blow-ups at ceiling points. The set of blow-ups at floor seam points is disjoint from the set of blow-up functions at ceiling seam points except along their shared boundary spheres.
\end{lemma}

\begin{proof}
The first statement of this lemma can be proved following a similar argument to the proof of the previous lemma on the blow-ups at ceiling seam points.

It remains to show that the set of principal blow-ups at floor seam points is disjoint from the set of principal blow-ups at ceiling seam points. As above, we consider partial functions involved in the blow-up functions. Note that $\phi^+_{(a,b)} = \phi^-_{(a',b')}$ if and only if $(a', b') = (-a, -b)$. Hence we might ask if the principal blow-up at $p = (a,b, 1/\lambda^2)$ might be equal to the principal blow-up at $-p = (-a, -b, -1/\lambda^2)$. We observed in the previous proof that $-p$ was in the subdomain of $\phi^+_{(a,b)}$ in the principal blow-up at $p$, while $p$ was not. Similarly, we could show that $p$ is in the subdomain of $\phi^-_{(-a,-b)}$ in the principal blow-up at $-p$, while $-p$ is not. Hence these two blow-up functions cannot be equal.
\end{proof}

\begin{proof}[Proof of Theorem B] 
We note that the family of all blow-up functions of interior ceiling points and ceiling seam points give a closed ball of dimension $2n$ with boundary sphere isomorphic to $\bhor(\R^{2n},\norm{\cdot}_{V_1})$. The same is true of the family of all blow-up functions of interior floor points and floor seam points.

When we identify the common boundary spheres of these two $(2n)$-balls, we get a set homeomorphic to $S^{2n}$. Finally, the blow-up functions at interior ceiling and floor give the same family, resulting in an identification of $\{\phi^+_{(a,b)} \mid \norm{(a,b)}_{V_1} \leq 1\}$ with $\{\phi^-_{(a,b)} \mid \norm{(a,b)}_{V_1} \leq 1\}$.
\end{proof}

\subsection{Non-separated example}\label{sec:deg_example}

In this section, we will describe in detail an example of a non-separated horofunction boundary of $\HHR$. Recall that non-separated here means that the two families of blow-up functions coming from the interior ceiling/floor points and the interior wall points of the metric sphere are not disjoint. This results in a horofunction boundary whose topology is different from the topology described in Theorem B.

Let $\norm{\cdot}_{V_1} = 2\norm{\cdot}_{\text{Eucl}}$ on $\R^{2}$. It is not a difficult exercise to show that for all $\lambda \in (0, 1]$, the norm $\norm{(x,y,z)}_\lambda = \max \{\norm{(x,y)}_{V_1},  \lambda\sqrt{|z|}\}$ satisfies the triangle inequality.

Hence we observe that the closure of the blow-ups at interior ceiling and floor points $\Phi = \{\frac{\lambda^2}4(-b\cdot x + a\cdot y)\mid \norm{(a,b)}_{V_1} \leq 1\}$ is the closed disk of radius $1/2$ centered at the origin in $(\R^2)^*$. Since $\norm{\cdot}_{V_1}$ is smooth, we recall that the family of blow-ups at interior wall points $\Delta = \bhor(\R^{2n}, \norm{\cdot}_{V_1}) = \partial Q_1^\circ$, which is the circle of radius $1/2$ in $(\R^2)^*$. Thus $\Phi$ and $\Delta$ intersect along the boundary of the disk of radius 1/2.

The spherical shell (i.e., annulus) of ceiling seam blow-up points $S^1 \times I$, therefore, has its boundary circles identified (with a rotation), giving us a torus of functions in the boundary. The floor blow-up functions contribute another torus, also attached to the disk $\Phi$ along its boundary, as in Figure~\ref{fig:degenerate}. This results in a horofunction boundary whose topological type is different from what was described in Theorem B and shown in Figure~\ref{fig:button}.

\begin{figure}[ht]
\centering
\begin{tikzpicture}[scale=.8]

\node at (0,-.2) {\includegraphics[width=1.8in]{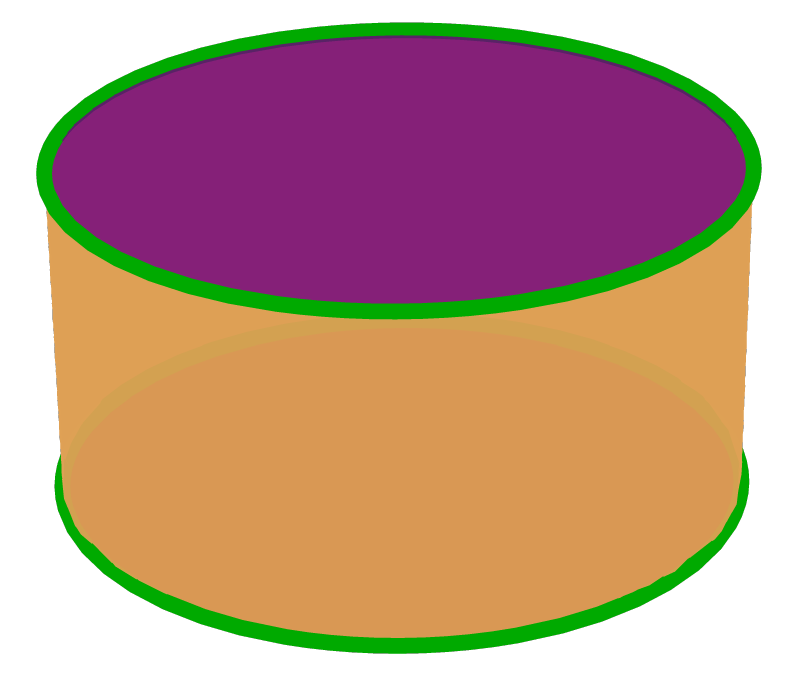}};

\begin{scope}[xshift = 10 cm, yshift = 0cm]

\node at (0,-.2) {\includegraphics[width=1.8in]{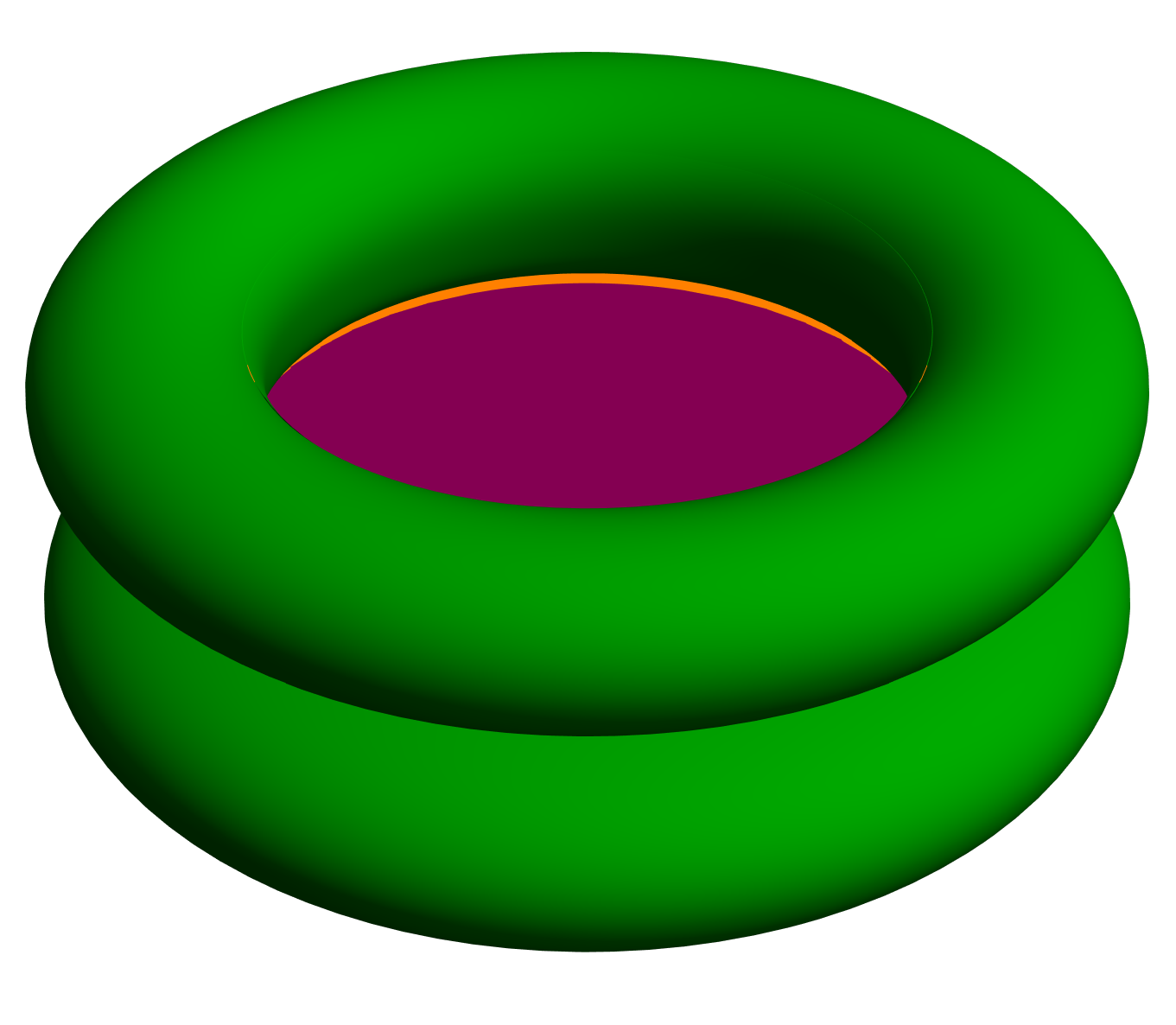}};

\end{scope}
\end{tikzpicture}
\caption{Unit metric sphere (left) and horofunction boundary (right) of $\HR$ where $\norm{\cdot}_{V_1}= 2\norm{\cdot}_{\text{Eucl}}$. In this non-separated example, the boundary has a different homeomorphism type.}\label{fig:degenerate}
\end{figure}

If we choose $\lambda < 1$, the intersection $\Psi \cap \Delta$ is empty, and the horofunction boundary of the homogenenous norm $\norm{(x,y,z)} = \max \{\norm{(x,y)}_{V_1}, \lambda \sqrt{|z|}\}$ on $\HHR$ is separated. That is, for this family of layered sup norms where $\norm{\cdot}_{V_1} = 2\norm{\cdot}_{\text{Eucl}}$ and $\norm{\cdot}_{V_2} = \lambda \sqrt{|\cdot|}$, for all but one value of $\lambda \in (0, 1]$, the horofunction boundary has the topology described in Theorem B.

For all other families of examples the author has considered, there has either been no value or a unique value of $\lambda$ (among those for which the homogenous norm satisfies the triangle inequality) for which the horofunction boundary is non-separated, suggesting that the topology of the boundary described in Theorem B is typical or generic among polysmooth layered sup norms.

\section{Horofunction boundaries of filiform Lie groups}\label{sec:filiform}

In this section we consider the $n$-dimensional, step-$(n-1)$ filiform Lie groups $L_n$ in exponential coordinates of the first kind for $n \geq 3$. As described above in Section~\ref{sec:exp-coords}, we can define $L_n$ by first defining its Lie algebra. Let $\mathfrak{l}_n$ be the $n$-dimensional Lie algebra generated by the vectors $X_1, X_2, \ldots, X_n$, where the only nontrivial bracket relations are
$[X_1,X_j] = X_{j+1}, \; 2\leq j \leq n-1$.
This Lie algebra admits the stratification 
$\g = \langle X_1, X_2\rangle \oplus \langle X_3\rangle \oplus \cdots \oplus \langle X_n\rangle$.

We can view $L_n$ as $\R^n$ with the group multiplication defined by the Baker--Campbell--Hausdorff formula. For any $n \geq 3$, the group multiplication in the first three coordinates is identical to that of the Heisenberg group $\HR = L_3$, and the group $L_4$ is commonly known as the Engel group. As $n$ gets large, the Baker--Campbell--Hausdorff formula involves more terms with higher-order commutators, and consequently, the group multiplication becomes quite complicated and difficult to write down. This complexity makes it difficult to give a full analytic description of the horofunction boundary, but we are able to determine the dimension of the horofunction boundary of $L_n$ equipped with any polysmooth layered sup norm for any $n \geq 3$.

\begin{thmc}\label{thm-filiform}
The filiform group $L_n$ when equipped with a polysmooth layered sup norm has a horofunction boundary which is of the expected dimension of $(n-1)$ when $n \leq 7$. For $n>7$, the horofunction boundary is of dimension $\lceil\frac{ n }{2}\rceil + 2$, which is strictly less than $n-1$.
\end{thmc}

To the author's knowledge, this theorem provides the first examples of topological groups $G$ whose horofunction boundaries are not of dimension $\dim(G)-1$. 

To prove Theorem~C, we will first show there are upper bounds on the dimension of the horofunction boundary of $L_n$ in Section~\ref{sec-upperbounds}. We will then exhibit families of horofunctions of maximal dimension for each $L_n$ in Section~\ref{sec-thmC}. First, though, we will make some observations concerning blow-ups and horofunctions in the boundary of $L_n$ in Section~\ref{sec-observations}, and in Section~\ref{sec-example}, we will look at parts of the horofunction boundary of $L_4$ to provide some concrete examples.

\subsection{Observations on blow-ups in $\bhor(L_n)$}\label{sec-observations}

Consider $L_n$ equipped with a polysmooth layered sup norm
$\norm{x} = \max_{1\leq j \leq s} \left\{\norm{\pi_j(x)}_{V_j}^{1/j}\right \}$. Recall that for the filiform Lie groups $L_n$, every layer of the stratification is 1-dimensional except for the first layer $V_1$, which is 2-dimensional. Thus any layered sup norm differs from the homogeneous sup norm only in the first layer of the group, if at all.

Using Lemma 2.3 from \cite{FNG}, we know that every horofunction in $\bhor(L_n)$ can be realized as the blow-up of the norm at some point $p \in \partial B$, the unit metric sphere for the norm. As we did for $\HHR$, we will partition $\partial B$ into faces. We define a \emph{face} of $\partial B$ not using convex geometry, but rather as a maximal connected set of points that lies either 
 \begin{enumerate}[(i)]
 \item entirely in a region $\Omega$ on which the layered sup norm is Pansu differentiable, or
 \item in the intersection of a finite collection of regions $\Omega_1, \ldots, \Omega_\ell$ on each of which the homogeneous norm is Pansu differentiable.
  \end{enumerate}
This is related to but not necessarily equivalent to saying a face is a connected set of points such that $\{\norm{\pi_i(p)}_{V_i} = 1 \mid i \in I\}$, for $I$ some nonempty subset of indices $I \subseteq \{1, \ldots, n-1\}$.

For example, consider the unit spheres of $L_3 = H(\R)$ in Figure~\ref{fig:button}. The ceiling, the floor, the walls, and the seams are all faces of $\partial B$. In the smooth case, there is a single smooth face coming from the vertical wall defined by $\{(x, y, z) : \norm{(x,y)}_{V_1} = 1, |z| < 1\}$. Meanwhile, for polygonal norms we can further subdivide the set $\{(x, y, z) : \norm{(x,y)}_{V_1} = 1, |z| < 1\}$ into the blue squares, which are faces of type (i), and the orange vertical edges, which are faces of type (ii).

Let $p \in \partial B$, and suppose that $\Omega$ is any closed neighborhood of $p$ such that $p$ is in the interior of $\Omega$. We denote by $\BU(p)$ the set of all blow-up functions of the norm, $\BU((\Omega, \norm{\cdot}), \{p_n\},\{\epsilon_n\})$, where $\{p_n\}$ is a sequence which converges to $p$. This set will be equal to the principal blow-up of the norm at $p$ along with any (finite or infinite) translates of that principal blow-up. If $\cF \subseteq \partial B$, we write $\BU(\cF) = \cup_{p \in \cF} \:\BU(p)$ to denote the set of all blow-ups along points in $\cF$.

After decomposing $\partial B$ into a finite collection of disjoint faces, we then have that
\[\dim \bhor(L_n) = \max_{\cF} \;\dim \BU (\cF).
\]

\begin{observation}\label{obs:translates} Theorem~A tells us that all of our horofunctions are continuous piecewise-linear functions of the coordinates $x_1$ and $x_2$. Additionally, a consequence of Proposition~\ref{prop5e8c4f94} is that in the principal blow-up of the norm at any non-smooth point on the unit sphere, each partial function will have subdomain a positive linear cone based at the origin.

Suppose we have a continuous piecewise-linear function in two variables defined with two distinct linear partial functions. By continuity, the boundary separating the subdomains of the two linear functions is a line, and there is a 1-dimensional family of translates of that principal blow-up satisfying $f(e) = 0$. Another consequence of Proposition~\ref{prop5e8c4f94} is that we will see all of these translates as blow-ups in the horofunction boundary. See Figure~\ref{fig:translates2}.

On the other hand, a continuous piecewise-linear function in two variables defined using three or more partial linear functions (where each partial function has subdomain equal to a positive linear cone) has a 2-dimensional family of translates satisfying $f(e) = 0$. See Figure~\ref{fig:translates3}.

\end{observation}

\begin{figure}[h]
\centering
\begin{tikzpicture}[scale = 1]

\draw (-2,0) -- (2,0);
\draw (0,-2) -- (0, 2);
\fill[magenta, opacity =.5] (-2,-.5) -- (-2,2) -- (.5,2) -- cycle;
\fill[cyan, opacity =.5] (-2,-.5) -- (-2, -2) -- (2,-2) -- (2,2) -- (.5, 2) -- cycle;
\draw[thick] (-2,-.5) -- (.5,2);
\node[gray!40!black, scale = .8, rotate = 45] at (-1, 1) {$\alpha x_1 + \beta x_2 - C$};
\node[gray!40!black, scale = .8] at (1, -1) {$\gamma x_1 + \delta x_2$};

\begin{scope}[xshift = 5.5cm]
\draw (-2,0) -- (2,0);
\draw (0,-2) -- (0, 2);
\fill[magenta, opacity =.5] (-2,-2) -- (-2,2) -- (2,2) -- cycle;
\fill[cyan, opacity =.5] (-2,-2) -- (2,-2) -- (2,2) -- cycle;
\draw[thick] (-2,-2) -- (2,2);
\node[gray!40!black, scale = .8] at (-1, 1) {$\alpha x_1 + \beta x_2$};
\node[gray!40!black, scale = .8] at (1, -1) {$\gamma x_1 + \delta x_2$};

\end{scope}

\begin{scope}[xshift = 11cm]
\draw (-2,0) -- (2,0);
\draw (0,-2) -- (0, 2);
\fill[magenta, opacity =.5] (-1, -2) -- (-2,-2) -- (-2,2) -- (2,2) -- (2, 1) -- cycle;
\fill[cyan, opacity =.5] (-1,-2) -- (2,-2) -- (2,1) -- cycle;
\draw[thick] (-1,-2) -- (2,1);
\node[gray!40!black, scale = .8] at (-1, 1) {$\alpha x_1 + \beta x_2$};
\node[gray!40!black, scale = .8, rotate = 45] at (1, -1) {$\gamma x_1 + \delta x_2 - D$};

\end{scope}

\end{tikzpicture}
\caption{Example of a principal blow-up function in two variables with two linear partial functions (center) along with two of its translates (left and right), where $f(e)$ always equals 0.}\label{fig:translates2}
\end{figure}

\begin{figure}[h]
\centering
\begin{tikzpicture}[scale = 1]
\draw (-2,0) -- (2,0);
\draw (0,-2) -- (0, 2);
\fill[magenta, opacity =.5] (-2,.5) -- (-.5,.5) -- (1,2) -- (-2,2) -- cycle;
\fill[cyan, opacity =.5] (-2,.5) -- (-.5,.5) -- (2, -2) -- (-2,-2) -- cycle;
\fill[orange, opacity =.5] (-.5,.5) -- (2, -2) -- (2, 2) -- (1, 2) -- cycle;
\draw[thick] (-2,.5) -- (-.5,.5) -- (1,2);
\draw[thick] (-.5, .5) -- (2, -2);
\node[gray!40!black, scale = .8] at (-.65, 1.6) {$\alpha x_1 + \beta x_2 - C$};
\node[gray!40!black, scale = .8] at (-1, -.7) {$\gamma x_1 + \delta x_2$};
\node[gray!40!black, scale = .8] at (1.1, .5) {$\epsilon x_1 + \zeta x_2$};

\begin{scope}[xshift = 5.5cm]
\draw (-2,0) -- (2,0);
\draw (0,-2) -- (0, 2);
\fill[magenta, opacity =.5] (-2,0) -- (0,0) -- (2,2) -- (-2,2) -- cycle;
\fill[cyan, opacity =.5] (-2,0) -- (0,0) -- (2,-2) -- (-2,-2) -- cycle;
\fill[orange, opacity =.5] (0,0) -- (2, -2) -- (2, 2) -- cycle;
\draw[thick] (-2,0) -- (0,0) -- (2,2);
\draw[thick] (0,0) -- (2, -2);
\node[gray!40!black, scale = .8] at (-1, 1) {$\alpha x_1 + \beta x_2$};
\node[gray!40!black, scale = .8] at (-1, -1) {$\gamma x_1 + \delta x_2$};
\node[gray!40!black, scale = .8] at (1.2, .2) {$\epsilon x_1 + \zeta x_2$};

\end{scope}

\begin{scope}[xshift = 11cm]
\draw (-2,0) -- (2,0);
\draw (0,-2) -- (0, 2);
\fill[magenta, opacity =.5] (-2,-.7) -- (-.3,-.7) -- (2,1.6) -- (2,2) -- (-2,2) -- cycle;
\fill[cyan, opacity =.5] (-2,-.7) -- (-.3,-.7) -- (1,-2) -- (-2,-2) -- cycle;
\fill[orange, opacity =.5] (-.3,-.7) -- (1, -2) -- (2, -2) -- (2, 1.6) -- cycle;
\draw[thick] (-2,-.7) -- (-.3,-.7) -- (2,1.6);
\draw[thick] (-.3,-.7) -- (1, -2);
\node[gray!40!black, scale = .8] at (-1, .7) {$\alpha x_1 + \beta x_2$};
\node[gray!40!black, scale = .8] at (-.6, -1.52) {$\gamma x_1 + \delta x_2 - D_1$};
\node[gray!40!black, scale = .8] at (1, -.7) {$\epsilon x_1 + \zeta x_2 - D_2$};

\end{scope}

\end{tikzpicture}
\caption{Example of a principal blow-up function in two variables with three linear partial functions (center) along with two of its translates (left and right), where $f(e)$ always equals 0.}\label{fig:translates3}
\end{figure}

\begin{observation} Recall that every layer of the stratification of $L_n$ is 1-dimensional except for the first layer $V_1$. Thus for each $i\geq 3$, we have a positive and negative $(n-1)$-dimensional face of points $p = (a_1, \ldots, a_n) \in \partial B$ defined by $a_i =1$ (or $a_i = -1$), where $\norm{(a_1,a_2)}_{V_1} < 1$, and  $|a_j| <1$ for all $j>2, j\neq i$. At points on these faces, the Pansu derivative of the layered sup norm is the Pansu derivative of the homogeneous coordinate function $|x_i|^{1/\nu_i}$. In the discussion after the proof of Proposition~\ref{prop:pansuDiff}, we observed that the Pansu derivatives of these homogeneous coordinate functions depend only on a coefficient $c_{i,1}$, which is the coefficient of $t$ in $(p\delta_tx)_i$, the $i$th coordinate of the product $p\delta_tx$. This coefficient $c_{i,1}$ is polynomial in the coordinates of $p$ and the coordinates of $x$ coming from the first layer. Hence, it suffices for us to specifically look at the coefficients of $tx_1$ and $tx_2$ in the $i$th coordinate of the the product $p\delta_tx$ as in Table~\ref{table:coords}.

Using exponential coordinates we let 
$p = X = a_1 X_1 + a_2 X_2 + \cdots + a_nX_n,$ and 
$\delta_tx = Y =  tx_1X_1 + tx_2X_2 + \cdots + t^{n-1}x_nX_n.$
A close look at the Baker--Campbell--Hausdorff formula and Proposition~\ref{prop:polynomials} tells us that terms in the product $p\delta_tx$ which only involve linear factors of $t$ must come from brackets which contain exactly one $Y$, i.e., of the form $[X, \ldots, X, Y]$. These brackets appear in the Baker--Campbell--Hausdorff formula with coefficients related to the Bernoulli numbers. For more detail see, for example, M\"uger's notes \cite{muger-notes}.

In Table~\ref{table:coords}, we give the terms in the $i$th coordinate of the product $p\delta_t x$ which are either constant or linear with respect to the variable $t$. All other terms of higher degree in $t$ do not impact the Pansu derivative of the homogeneous coordinate functions and so are not needed.

\renewcommand{\arraystretch}{1.2}
\begin{table}[H]
$
\begin{array}{|c|c|c|}
\hline
\text{Coordinate $i$} & \text{$i$th coordinate $(p\delta_tx)_i$ of product $(p\delta_tx)$}\\
\hline
1 & a_1 + tx_1 \\
2 & a_2 + tx_2 \\
3 & a_3 + tx_1(-\frac12a_2) + tx_2(\frac12 a_1) + \dots\\
i\geq 4 & 
\begin{cases}
a_i + tx_1(-\frac12a_{i-1} - \frac1{12}a_1 a_{i-2} - \cdots - \frac{B_{2k-2}}{(2k-2)!}a_1^{2k-3}a_2) + tx_2(\frac{B_{2k-2}}{(2k-2)!}a_1^{2k-2}) + \dots, & i = 2k\\
a_i + tx_1(-\frac12a_{i-1} - \frac1{12}a_1 a_{i-2} - \cdots - \frac{B_{2k-2}}{(2k-2)!}a_1^{2k-3}a_3) + \dots, & i = 2k+1
\end{cases}\\
\hline
\end{array}$
\caption{Multiplication in $L_n$.}
\label{table:coords}
\end{table}

In Table~\ref{table:pD}, we give the Pansu derivatives of the homogeneous coordinate functions making use of the formula found in the discussion after the proof of Proposition~\ref{prop:pansuDiff}. We note that in the filiform group $L_n$, the dilation weight or homogeneous degree $\nu_i$ of the coordinate function $x_i$ is
\[
\nu_i = \begin{cases}
1, & i = 1,2\\
i-1 & i \geq 3
\end{cases}.
\]

\begin{table}[ht]
$
\begin{array}{|c|c|}
\hline
\text{Coordinate $i$} & \text{Pansu derivative of $i$th homogeneous coordinate function $|x_i|^{1/\nu_i}$, $x_i>0$}\\
\hline
1 &  x_1\\
2 & x_2\\
3  & \frac{1}{2}\left[-\frac12a_2x_1 + \frac12a_1x_2\right]\\
i\geq 4 & 
\begin{cases}
\frac{1}{i-1}\left[(-\frac12a_{i-1} - \frac1{12}a_1 a_{i-2} - \cdots - \frac{B_{2k-2}}{(2k-2)!}a_1^{2k-3}a_2)x_1 + \frac{B_{2k-2}}{(2k-2)!}a_1^{2k-2}x_2\right], & i = 2k\\
\frac{1}{i-1}\left[(-\frac12a_{i-1} - \frac1{12}a_1 a_{i-2} - \cdots - \frac{B_{2k-2}}{(2k-2)!}a_1^{2k-3}a_3)x_1\right], & i = 2k+1\\
\end{cases}\\
\hline
\end{array}$
\caption{Pansu derivatives of homogeneous coordinate functions.}
\label{table:pD}
\end{table}
\end{observation}

\begin{observation} \label{obs:injection}

Suppose that $\cF$ is an $m$-dimensional face of $\partial B$ with free parameters $a_1, \ldots, a_m$. Additionally suppose that at each point $p \in \cF$, the principal blow-up of the polysmooth layered sup norm is defined by $\ell$ partial functions, each of which is a linear function of $x_1$ and $x_2$ with coefficients which are (possibly constant) polynomials in the free parameters. Let $G$ be the list of these polynomial coefficients. Then $|G| \leq 2\ell$, and the family of all possible principal blow-ups at points $p$ along $\cF$ has dimension less than or equal to $\min\{m, 2\ell\}$. Indeed, the dimension of this family of principal blow-ups is equal to the size of the largest injective partial map
\[
\{a_1, \ldots, a_k\} \to G,
\]
where each $a_i$ is sent to a coefficient in $G$ in which $a_i$ appears nontrivially.

Imagine, for example, that $\cF$ is a 4-dimensional face of $\partial B$ parametrized by four free parameters, $a_{1}$, $a_{2}$, $a_{3}$, and $a_{4}$. Let us further suppose that for each point $p \in \cF$, the principal blow-up of the polysmooth layered sup norm at $p$ is made up of three partial functions. That is, we're assuming that $\cF$ is defined by points which lie in the intersection of three regions $\Omega_1, \Omega_2,$ and $\Omega_3$ on each of which the layered sup norm is Pansu differentiable.

In this case, the principal blow-up at each $p \in \cF$ will look something like the center panel of Figure~\ref{fig:translates3}, and so $G = \{\alpha, \beta, \gamma, \delta, \epsilon, \gamma\}$. Each of these coefficients in $G$ is a polynomial in the free parameters $a_{1}$, $a_{2}$, $a_{3}$, and $a_{4}$. As we vary the four free parameters, if we want the family of principal blow-ups along $\cF$ to be 4-dimensional, we need there to be an injective map from $\{a_{1}, a_{2}, a_{3}, a_{4}\}$ to $G$, where each $a_{i}$ is sent to a coefficient in which $a_{i}$ appears nontrivially. Examples~\ref{ex:inconsistency} and \ref{ex:injection} give explicit computations where we fail to have this type of injective map.
\end{observation}

\subsection{An example: Parts of the horofunction boundary of $L_4$}\label{sec-example}

%For $n=3$, we have that $L_3 = H(\R)$, the 3-dimensional Heisenberg group, and we can use the results from the previous section to find the blow-ups of any polysmooth layered sup norm along the ceiling/floor, walls, and seams of the unit sphere.

Let us consider the 4-dimensional filiform group $L_4$ with group multiplication
\begin{align*}(x_1,x_2,x_3,x_4)(y_1,y_2,y_3,y_4) = \Bigl( x_1+y_1, x_2+y_2&, x_3+y_3+\frac12(x_1y_2 - x_2y_1), \\
& x_4 + y_4 +\frac12(x_1y_3-x_3y_1)+ \frac1{12}(x_1-y_1)(x_1y_2-x_2y_1)\Bigr).
\end{align*}

We will equip $L_4$ with the homogeneous sup norm $\norm{(x_1,x_2,x_3,x_4)}_\infty = \max\{|x_1|, |x_2|, \sqrt{|x_3|}, \sqrt[3]{|x_4|}\}$, and we will compute the families of blow-ups one finds along a few different faces of the unit sphere, i.e., cube, to connect the observations from the previous section to this concrete example.

For simplicity, we will restrict our attention to the all positive orthant of $\R^4 \simeq L_4$, and we can partition this orthant into closed regions on which the homogeneous sup norm takes on different values. For example, there is a region within this orthant on which $x_1$ ``wins'', i.e., $\norm{(x_1,x_2,x_3,x_4)}_\infty = \max\{x_1, x_2, \sqrt{x_3}, \sqrt[3]{x_4}\} = x_1$. Indeed, this will occur in the region given by
\[
\Omega_1 = \{(x_1,x_2,x_3,x_4) \mid x_1 - x_2 \geq 0, \; x_1^2 - x_3 \geq 0,\; x_1^3 - x_4 \geq 0 \}.
\]
We can similarly define regions on which $x_2$, $\sqrt{x_3}$, and $\sqrt[3]{x_4}$ are picked out by the max function in the positive orthant, given, respectively, by
\begin{align*}
\Omega_2 &= \{(x_1,x_2,x_3,x_4) \mid x_2 - x_1 \geq 0, \; x_2^2 - x_3 \geq 0, \; x_2^3 - x_4 \geq 0 \},\\
\Omega_3 &= \{(x_1,x_2,x_3,x_4) \mid x_3 - x_1^2 \geq 0, \; x_3 - x_2^2 \geq 0, \; x_3^{3/2} - x_4 \geq 0 \},\\
\Omega_4 &= \{(x_1,x_2,x_3,x_4) \mid x_4 - x_1^3 \geq 0, \; x_4 - x_2^3 \geq 0,\;  x_4- x_3^{3/2} \geq 0 \}.
\end{align*}

\begin{ex}\textbf{Family of blow-ups at the vertex $(1,1,1,1)$}

We note that the point $p = (1,1,1,1)$ is a point on the unit sphere of the homogeneous sup norm, and it is on the boundary of all four regions $\Omega_i$, $1\leq i \leq 4$. To compute the principal blow-up of $\norm{\cdot}_\infty$ at $p$, we want to understand the behavior of $\lim_{t\to0} \frac{\norm{p\delta_tx}_\infty - \norm{p}_\infty}{t}$. As $x= (x_1,x_2,x_3,x_4)$ varies, $p\delta_t x$ will visit all four $\Omega_i$, affecting the way the norm is computed. Thus the principal blow-up will be composed of four partial functions, each a linear function of $x_1$ and/or $x_2$.

Fortunately Propositions~\ref{prop5e8c4f94} and \ref{prop5e8c979c} tell us how to compute this principal blow-up. In Table~\ref{table:pD} we already have the Pansu derivatives of the homogeneous coordinate functions computed. It remains to find the subdomain, in this case the linear cone, on which each of these partial functions is defined. Again we turn to Pansu derivatives, as in Proposition~\ref{prop5e8c4f94}. For the region $\Omega_1$, the Pansu derivative of the homogeneous coordinate function $|x_1|$ is simply $x_1$, and $\Omega_1$ is defined by $F_1(x) = x_1 - x_2 \geq 0$, $F_2(x) = x_1^2 - x_3 \geq 0$, and $F_3(x) = x_1^3 - x_4 \geq 0$. We compute the Pansu derivative of $F_1$ at $p$ to find
\[
\pD F_1\vert_p(x) = \lim_{t\to0}\frac{F_1(p\delta_tx) - F_1(p)}{t} = \lim_{t\to0}\frac{(1+tx_1) - (1+tx_2)}{t} = x_1-x_2. 
\]
Similarly, we find 
$ \pD F_2\vert_p(x) = \frac52 x_1 - \frac12 x_2,
\text{ and } \pD F_3\vert_p(x) = \frac{43}{12} x_1 - \frac{1}{12} x_2.$
Thus in the principal blow-up at $p = (1,1,1,1)$, the partial function $x_1$ is defined on the linear cone given by $\cap_{i=1}^3 F_i(x) \geq 0 = \{(x_1,x_2) \mid x_2 \leq x_1,\; x_2 \leq 5x,\; x_2 \leq 43x \}$, as shown in Figure~\ref{fig:example1}. The condition $x_2 \leq 5x$ is redundant here, but still had to be computed.

\begin{figure}[h]
\centering
\begin{tikzpicture}[scale = 1]

\draw[gray] (-2,0) -- (2,0);
\draw[gray] (0,-2) -- (0, 2);
\fill[magenta, opacity =.5] (-1/20,-2) -- (0,0) -- (2,2) -- (2, -2) -- cycle;
\fill[cyan, opacity =.5] (2,2) -- (0,0) -- (-2,2/3) -- (-2, 2) -- cycle;
\fill[orange, opacity =.5] (-2,2/3) -- (0,0) -- (-2,-1/2) -- cycle;
\fill[green!40!gray, opacity =.5] (-2,-1/2) -- (0,0) -- (-1/20, -2) -- (-2,-2) -- cycle;
\draw[thick] (0,0)--(2,2);
\draw[thick] (0,0) -- (-1/20, -2);
\draw[thick] (0,0) -- (-2,-1/2);
\draw[thick] (0,0) -- (-2, 2/3);
\node[gray!40!black, scale = .8] at (1, -1) {$x_1$};
\node[gray!40!black, scale = .8] at (-.5, 1) {$x_2$};
\node[gray!40!black, scale = .8] at (-3, 0) {$\frac{-1}{4}x_1 + \frac14 x_2$};
\draw[thick, gray!40!black, ->] (-2.2, 0) to[bend left = 20] (-1.8, .2);
\node[gray!40!black, scale = .8] at (-1, -1) {$\frac{-7}{36}x_1 + \frac1{36} x_2$};

\end{tikzpicture}
\caption{Principal blow-up at the vertex $(1,1,1,1)$.}\label{fig:example1}
\end{figure}

A similar analysis for the three other regions $\Omega_2$, $\Omega_3$, and $\Omega_4$ tells us that the principal blow-up of the homogeneous sup norm at $p = (1,1,1,1)$ is
\[
f(x) = \begin{cases}
x_1, & x_2 \leq x_1 \text{ and } x_2 \leq 43x_1\\
x_2, & x_2 \geq x_1 \text{ and } x_2 \geq -1/3x_1\\
-1/4x_1 + 1/4x_2, &  1/4x_1 \leq x_2 \leq -1/3x_1 \\
-7/36x_1 + 1/36x_2, &  43x_1 \leq x_2 \leq 1/4x_1
\end{cases}.
\]

As discussed above in Observation~\ref{obs:translates}, since this principal blow-up is composed of at least three partial functions, it has a two-dimensional family of translates. We see, therefore, that $\BU(p)$ is a 2-dimensional family of continuous functions satisfying $f(e) = 0$.

\end{ex}

\begin{ex}\textbf{``Inconsistency'' in the dimension of families of blow-ups}\label{ex:inconsistency}

Here we will exhibit two 2-dimensional faces of the unit sphere of the homogeneous sup norm on $L_4$ whose respective families of blow-ups are of different dimensions.

Let us first consider the set $\mathcal F_1 = \{(a_1,a_2,1,1)\mid 0< a_1,a_2 <1\}$, which is the intersection of a 2-dimensional face with the positive orthant. Points $p \in \mathcal F_1$ lie along the boundary of the regions $\Omega_3$ and $\Omega_4$ defined above, and as such the principal blow-up at any point $p \in \cF_1$ will be defined using two partial functions. From Table~\ref{table:pD}, we see that the two corresponding partial functions (i.e., Pansu derivatives) are $f_1(x) = -a_2/4\:x_1 + a_1/4\:x_2$ and $f_2(x) = (-1/6 - a_1a_2/36) x_1 + a_1^2/36\:x_2$. The boundary between $\Omega_3$ and $\Omega_4$ in $L_4$ is given by $F(x) = x_4 - x_3^{3/2} = 0$, and the Pansu derivative of $F$ gives us the boundary between $f_1$ and $f_2$ in the blow-up. As shown in the left panel of Figure~\ref{fig:example2}, we end up with
\[
f(x) = \begin{cases}
-a_2/4\:x_1 + a_1/4\:x_2, & x_2 \geq (6-9a_2 + a_1a_2)/(-9a_1 + a_1^2) x_1\\
(-1/6 - a_1a_2/36) x_1 + a_1^2/36\:x_2, & x_2 \leq (6-9a_2 + a_1a_2)/(-9a_1 + a_1^2) x_1
\end{cases}.
\]

\begin{figure}[h]
\centering
\begin{tikzpicture}[scale = 1]

\draw (-2,0) -- (2,0);
\draw (0,-2) -- (0, 2);
\fill[magenta, opacity =.5] (-2,-1) -- (-2,2) -- (2,2) -- (2, 1) -- cycle;
\fill[cyan, opacity =.5] (-2,-1) -- (-2, -2) -- (2,-2) -- (2,1) -- cycle;
\draw[thick] (-2,-1) -- (2,1);
\node[gray!40!black, scale = .8] at (-1, 1) {$-\frac{a_2}4 x_1 + \frac{a_1}4x_2$};
\node[gray!40!black, scale = .8] at (.55, -1) {$\left(-\frac16 - \frac{a_1a_2}{36}\right) x_1 + \frac{a_1^2}{36}x_2$};

\begin{scope}[xshift = 7cm]
\draw (-2,0) -- (2,0);
\draw (0,-2) -- (0, 2);
\fill[magenta, opacity =.5] (-2,-2) -- (-2,2) -- (.5,2) -- (-.5, -2) -- cycle;
\fill[cyan, opacity =.5] (-.5,-2) -- (2,-2) -- (2,2) -- (.5, 2) -- cycle;
\draw[thick] (-.5,-2) -- (.5,2);
\node[gray!40!black, scale = .8] at (-1, 1) {$-\frac{a_2}4 x_1 + \frac14x_2$};
\node[gray!40!black, scale = .8] at (1, -1) {$x_1$};

\end{scope}

\end{tikzpicture}
\caption{Left: Principal blow-up at $(a_1,a_2, 1,1)$. Right: Principal blow-up at $(1,a_2, 1, a_4)$.}\label{fig:example2}
\end{figure}

As we vary $a_1$ and $a_2$ we get a 2-dimensional family of principal blow-up functions, and since each principal blow-up is composed of two partial functions, each one has a 1-dimensional family of translates. In the end, this two-dimensional face $\cF_1$ of $\partial B$ gives us a 3-dimensional family of blow-up functions.

Now let us consider the two-dimensional face in the positive orthant given by
$\mathcal F_2 = \{(1,a_2,1,a_4)\mid 0< a_2, a_4 <1\}$. Points $p \in \mathcal F_2$ lie along the boundary of the regions $\Omega_1$ and $\Omega_3$, and so again each principal blow-up at a point $p \in \cF_2$ will be defined via two partial functions. From Table~\ref{table:pD}, we see that the two corresponding partial functions (i.e., Pansu derivatives) are $g_1(x) = x_1$ and $g_2(x) = -a_2/4\:x_1 + 1/4\:x_2$. As shown in the right panel of Figure~\ref{fig:example2}, the principal blow-up at $p$ is
\[
g(x) = \begin{cases}
x_1, & x_2 \leq (4+a_2)x_1\\
-a_2/4\:x_1 + a_1/4\:x_2, & x_2 \geq (4+a_2)x_1
\end{cases}.
\]

We observe here that the free parameter $a_4$ does not appear in the coefficients of either of the two partial functions. Thus as we vary $a_2$ and $a_4$, we only get a 1-dimensional family of principal blow-up functions. Along with the translates, we see that this two-dimensional set $\cF_2$ of $\partial B$ gives not a 3-dimensional, but rather 2-dimensional, family of blow-up functions.

\end{ex}

\begin{ex}\textbf{Injectivity of free parameters to coefficients of partial functions}\label{ex:injection}

Consider the positive face of the unit sphere given by $\norm{\pi_3(p)}_{V_3} = 1$. Intersected with the positive orthant, we have the 3-dimensional set of points $\cF =\{(a_1,a_2,a_3,1)\mid 0< a_1, a_2, a_3 <1\}$. All of $\cF$ is contained in the interior of $\Omega_4$, and so the principal blow-up of the homogeneous sup norm here is linear. From Table~\ref{table:pD}, we see that at the point $p = (a_1, a_2, a_3, 1)$, we get the principal blow-up
\[
h(x) = \left(-\frac{a_3}{6} - \frac{a_1a_2}{36}\right) x_1 + \frac{a_1^2}{36}x_2.
\]
While we have three free parameters, as we let $a_1$, $a_2$, and $a_3$ vary, we only get a two-dimensional family of principal blow-up functions. This example helps us see that if we would like the family of principal blow-ups to have dimension equal to $\dim \cF$, we need an injection between the set of free parameters in the family $\cF$ and the set of coefficients $G$, as described in Observation~\ref{obs:injection}. 

Since these principal blow-up functions are linear, there are no translates. Thus from this 3-dimensional face of $\partial B$, we get a 2-dimensional family of blow-up functions.
\end{ex}

\subsection{Upper bounds on the dimension of $\bhor(L_n)$}\label{sec-upperbounds}

\begin{lemma}\label{lem:n-1}
For $n\geq 3$, the horofunction boundary of $L_n$ equipped with a polysmooth layered sup norm has dimension less than or equal to $n-1$.
\end{lemma}

\begin{proof}
%From Proposition~\ref{prop:dimension} we know that $\dim \overline{L_n}^h = \dim L_n = n$. Thus, as a subspace of $\overline{L_n}^h$, we know that the horofunction boundary $\bhor(L_n)$ must also have dimension less than or equal to $n$. '

To prove this lemma, we will show that no face of the unit sphere $\partial B$ can have a family of blow-up functions of dimension greater than $n-1$.

Suppose that we have a face $\cF_1$ along which all principal blow-ups are linear functions. Then for each point $p \in \cF_1$, the set of blow-ups $\BU(p)$ is a single function (with no translates). In $L_n$, the unit metric ball has no faces with dimension greater than $n-1$, and so $\dim \BU(\cF_1) \leq n-1$.

Next, suppose that we have a face or subset of a face $\cF_2$ for which the principal blow-up at each point is made up of two partial functions. Then for each $p \in \cF_2$, we have that $\BU(p)$ is 1-dimensional since we have a one-dimensional family of translates of the principal blow-up. It remains to be shown that such a face $\cF_2$ can have dimension no greater than $n-2$. We consider three cases for points $p= (a_1, \ldots, a_n) \in \cF_2$. 
\begin{itemize}
\item \underline{Case 1:} Suppose that points $p \in \cF_2$ satisfy $\norm{\pi_1(p)}_{V_1} = 1$, that $\norm{\cdot}_{V_1}$ is polygonal, and that $\pi_1(p)$ is a vertex of the polygonal norm. Since $\pi_1(p)$ is a vertex, we have $\pi_1(p) = (a_1, a_2)$ is fixed, meaning we have maximum $(n-2)$ free parameters in $\cF_2$.

\item \underline{Case 2:} Suppose that points $p \in \cF_2$ satisfy $\norm{\pi_1(p)}_{V_1} = 1$, and that either $\norm{\cdot}_{V_1}$ is smooth or that $\norm{\cdot}_{V_1}$ is polygonal and the points $\pi_1(p)$ lie along an edge of the polygon. In this case, we have one free parameter in the first layer $V_1$, and the blow-up corresponding to the first layer is linear. Since the principal blow-up at $p$ is made up of two partial functions, it must be the case that $\norm{\pi_i(p)}_{V_i} = 1$ for another index $i \in \{2, \ldots, n-1\}$. Thus we have only $(n-2)$ free parameters in $\cF_2$.

\item \underline{Case 3:} Otherwise, the principal blow-up at $p \in \cF_2$ has two partial functions if $\norm{\pi_i(p)}_{V_i} = 1$ for two different indices in $\{2, \ldots, n-1\}.$ Thus two coordinates are fixed, and there are only $(n-2)$ free parameters in $\cF_2$. 
\end{itemize}

Finally, suppose we have a face or subset of a face $\cF_3$ for which the principal blow-up at each point is made up of at least three partial functions. Then for each $p \in \cF_3$, we have that $\BU(p)$ is 2-dimensional since we have a two-dimensional family of translates of the principal blow-up. It remains to be shown that such a face $\cF_3$ can have dimension no greater than $n-3$. We similarly consider three cases for points $p= (a_1, \ldots, a_n) \in \cF_3$.
\begin{itemize}
\item \underline{Case 1:} Suppose that points $p \in \cF_3$ satisfy $\norm{\pi_1(p)}_{V_1} = 1$, that $\norm{\cdot}_{V_1}$ is polygonal, and that $\pi_1(p)$ is a vertex of the polygonal norm. Then for the principal blow-up at $p$ to have at least three partial functions, we need $\norm{\pi_i(p)}_{V_i} = 1$ for at least one other index $i \in \{2, \ldots, n-1\}$, leaving us with no more than $(n-3)$ free parameters.

\item \underline{Case 2:} Suppose that points $p \in \cF_3$ satisfy $\norm{\pi_1(p)}_{V_1} = 1$, and that either $\norm{\cdot}_{V_1}$ is smooth or that $\norm{\cdot}_{V_1}$ is polygonal and the points $\pi_1(p)$ lie along an edge of the polygon. In this case, we have one free parameter in the first layer $V_1$, and the blow-up corresponding to the first layer is linear. Since the principal blow-up at $p$ is made up of three partial functions, it must be the case that $\norm{\pi_i(p)}_{V_i} = 1$ for at least two indices $i \in \{2, \ldots, n-1\}$. Thus we have no more than $(n-3)$ free parameters in $\cF_3$.

\item \underline{Case 3:} Otherwise the principal blow-up at $p \in \cF_3$ has at least partial functions if $\norm{\pi_i(p)}_{V_i} = 1$ for at least three different indices $i \in \{2, \ldots, n\}.$ Thus there are no more than $(n-3)$ free parameters in $\cF_3$. 
\end{itemize}
\end{proof}

\begin{lemma}
For $n\geq 3$, the horofunction boundary of $L_n$ equipped with a polysmooth layered sup norm has dimension less than or equal to $\lceil\frac{ n }{2}\rceil + 2$.
\label{lemma:otherBound}
\end{lemma}

\begin{proof}

To show that the horofunction boundary of $L_n$ is has dimension less than or equal to $\lceil\frac{ n }{2}\rceil + 2$, it suffices to show that for any face $\cF$ of any dimension in the unit metric sphere of the polysmooth layered sup norm, the family of blow-ups $\BU(\cF)$ has dimension bounded by $\lceil\frac{ n }{2}\rceil + 2$.

Suppose that $\cF$ satisfies $\norm{\pi_i(p)}_{V_i} = 1$ for all $i \in I \subseteq\{1, \ldots, n-1\}$, where the size of the index set $|I|$ is $k>0$. Given that $\norm{\cdot}_{V_1}$ could be polygonal or smooth, as in the proof of Lemma~\ref{lem:n-1}, there are a few cases to consider.

\begin{itemize}
\item \underline{Case 1:} Suppose that points $p= (a_1, \ldots, a_n) \in \cF$ satisfy $\norm{\pi_1(p)}_{V_1} = 1$, that $\norm{\cdot}_{V_1}$ is polygonal, and that $\pi_1(p)$ is a fixed vertex of the polygonal norm. Then since $|I| = k$, the points in $\cF$ have $k+1$ fixed coordinates, meaning the face is $m = (n-k-1)$-dimensional. The principal blow-ups at these points will be made up of $\ell = k+1$ partial functions.

\item \underline{Case 2:} Combining Cases 2 and 3 from the proof of the previous lemma, suppose that i) $I \subseteq \{2, \ldots, n-1\}$; or ii) points $p= (a_1, \ldots, a_n) \in \cF$ satsify $\norm{\pi_1(p)}_{V_1} = 1$, and that either $\norm{\cdot}_{V_1}$ is smooth or that $\norm{\cdot}_{V_1}$ is polygonal and the points $\pi_1(p)$ lie along an edge of the polygon. In this case, the points in $\cF$ have $k$ fixed coordinates, meaning the face is $m = (n-k)$-dimensional. The principal blow-ups at these points will be made up of $\ell = k$ partial functions.
\end{itemize}

Regardless of which case we are in, we have $m + \ell = n$, where $m$ is the number of free parameters in the face $\cF$ and $\ell$ is the number of partial functions in the principal blow-up at a point $p \in \cF$.
Let $\cP$ be list of Pansu derivatives which appear as partial functions of the principal blow-up at points $p \in \cF$. The list $\cP$ will consist of the Pansu derivatives of the homogeneous coordinate functions in Table~\ref{table:pD} for all $i \in I$, $i \geq 2$. If $1 \in I$, then depending on $\norm{\cdot}_{V_1}$, the list $\cP$ will contain one or two other linear functions in the variables $x_1$ and $x_2$ with coefficients which depend on $a_1$ and $a_2$ \cite{schilling-thesis}.
If we can define an injective map between the $m$ free parameters and the set $G$ of polynomial coefficients of $x_1$ or $x_2$ from the list $\cP$ (such that each free parameter $a_i$ is sent to a coefficient it appears in nontrivially), then we will have an $m$-dimension family of principal blow-ups associated to $\cF$.

We note, however, that the coefficients of $x_2$ in the Pansu derivatives of the coordinate functions seen in Table~\ref{table:pD} are always functions solely of $a_1$. This obstacle means no free parameter $a_i$, $i\neq 1$ can be mapped to a coefficient of $x_2$. Hence the dimension of the family of principal blow-ups along the face $\cF$ (where each principal blow-up has $\ell$ partial functions) is bounded above by $\ell+1$. To achieve this maximal dimension, we would need the number of free parameters $m$ to be at least $\ell + 1$, and so $(\ell + 1) + \ell \leq m + \ell = n$.

Thus to get an $(\ell+1)$-dimensional family of principal blow-ups (before considering translates), we need $\ell \leq \frac{n-1}{2}$. As long as $\ell \geq 3$, there will be a 2-dimensional family of translates for each principal blow-up, and hence the whole family of blow-ups $\BU(\cF)$ has dimension bounded by 
\[
(\ell + 1) + 2 \leq \left\lfloor\frac{n-1}{2}\right\rfloor + 3 = \left\lceil \frac{n}2\right\rceil + 2. \qedhere
\]
\end{proof}

\subsection{Proof of Theorem~C}\label{sec-thmC}

Given the upper bounds on the dimension of $\bhor(L_n)$ given above, we note that for $3 \leq n \leq 7$, we have that $(n-1) \leq \lceil\frac{ n }{2}\rceil + 2$. For $n > 7$, we have $\lceil\frac{ n }{2}\rceil + 2 < n-1$. To prove Theorem C, therefore, it suffices to find families of horofunctions of the corresponding dimensions in $\bhor(L_n)$ for each $n$.

\begin{lemma}\label{lem:fill}
Let $L_n$ be equipped with a polysmooth layered sup norm. There exist families of horofunctions in $\bhor(L_n)$ which are of dimension $\begin{cases}
n-1, & 3\leq n \leq 7\\
\lceil\frac{ n }{2}\rceil + 2, & n>7
\end{cases}$.
\end{lemma}

\begin{proof}

In Table~\ref{table:faces}, we provide a face $\cF_n$ of dimension $\lceil \frac{n}{2}\rceil$ of the unit sphere $\partial B$ for each $n \geq 3$. Note that we never have $\norm{\pi_1(p)}_{V_1} = 1$, and so we need not be concerned with whether $\norm{\cdot}_{V_1}$ is polyhedral or smooth since the Pansu derivatives associated to the first layer will not appear in our blow-ups. Thus each of the Pansu derivatives that appear as partial functions of the principal blow-ups at points $p \in \cF$ will be the Pansu derivatives of the homogeneous coordinate functions for $i \geq 3$, given in Table~\ref{table:pD}.

\begin{table}[H]
\centering
$
\begin{array}{|c|c|}
\hline
n & \text{Face to be considered}\\
\hline
3 & \{(a_1,a_2, 1): \norm{(a_1,a_2)}_{V_1}< 1\}\\
4 & \{(a_1, a_2, 1, 1): \norm{(a_1,a_2)}_{V_1}< 1\}\\
5  & \{(a_1, a_2, 1, a_4, 1): \norm{(a_1,a_2)}_{V_1}< 1, \norm{a_4}_{V_3} < 1\}\\
6 & \{(a_1, a_2, 1, a_4, 1, 1): \norm{(a_1,a_2)}_{V_1}< 1, \norm{a_4}_{V_3} < 1\}\\
%7 & \{(a_1, a_2, 1, a_4, 1,a_6, 1)\mid \norm{(a_1,a_2)}_{V_1}< 1, \norm{a_4}_{V_3} < 1,\norm{a_6}_{V_5} < 1\}\\
%8 & \{(a_1, a_2, 1, a_4, 1,a_6, 1, 1)\mid \norm{(a_1,a_2)}_{V_1}< 1, \norm{a_4}_{V_3} < 1,\norm{a_6}_{V_5} < 1\}\\
n = 2\ell +1 , \: \ell\geq 3 &  \{(a_1, a_2, 1, a_4, 1, \ldots, a_{n-3}, 1, a_{n-1},1): \norm{(a_1,a_2)}_{V_1}< 1, \norm{a_{2i}}_{V_{2i-1}} < 1, 2 \leq i \leq \ell\}\\
n = 2(\ell+1), \: \ell\geq 3 & \{(a_1, a_2, 1, a_4, 1, \ldots, a_{n-4}, 1, a_{n-2},1,1): \norm{(a_1,a_2)}_{V_1}< 1, \norm{a_{2i}}_{V_{2i-1}} < 1, 2 \leq i \leq \ell\}\\
\hline
\end{array}$
\caption{A face of the unit metric sphere in $L_n$ which gives an $(n-1)$-dimensional family of blow-ups.}
\label{table:faces}
\end{table}

For each $n$, let $J_n \subseteq \{3,\ldots, n\}$ be nonempty subset of indices such that $\cF_n$ is defined by $\{x_j = 1 \mid j \in J\}$. Let $\cP_n$ be the list of Pansu derivatives $\cP_n = \{\pD |x_j|^{1/\nu_j} \mid j\in J_n\}$ which appear as partial functions in the principal blow-up at points $p \in \cF_n$.

We note that for each face $\cF_n$ in Table~\ref{table:faces}, we can construct an injective map from the set of free parameters of that face to the set $G$ of polynomial coefficients of $x_1$ or $x_2$ in $\mathcal P_n$ such that each free parameter is sent to a coefficient in which it appears nontrivially. Thus the family of principal blow-ups along $\cF_n$ has dimension equal to the number of free parameters, namely $\lceil \frac{n}{2}\rceil$.

Indeed, for all $n \geq 3$, we have $3 \in J_n$, meaning for all $p \in \cF_n$, the Pansu derivative $\frac{1}{2}\left[-\frac12a_2x_1 + \frac12a_1x_2\right]$ appears as the principal blow-up at $p$ (for $n=3$) or as a partial function of the principal blow-up at $p$ (for $n > 3$). We can map the free parameters $a_1$ and $a_2$ to the coefficients of $x_2$ and $x_1$, respectively. 
Next, for all $n$, all other free parameters are of even index. In Table~\ref{table:faces}, for each even index $k$ such that $a_k$ is a free parameter, we have $a_{k+1}$ fixed at $1$. Thus $k+1 \in J_n$, and since $a_k$ appears in the coefficient of $x_1$ in the Pansu derivative of $|x_{k+1}|^{1/\nu_{k+1}}$, we can map $a_k$ there. Thus the family of principal blow-ups along $\cF_n$ has dimension $\lceil \frac{n}{2}\rceil$.

For $n=3$, the family of principal blow-ups along $\cF_3$ has dimension $\lceil \frac{3}{2}\rceil = 2$. We note that $|J_3| = 1$, and so our principal blow-ups are linear. There are no translates of these linear blow-ups, and so $\dim \BU(\cF_3) = 2$.

For $n=4$, the family of principal blow-ups along $\cF_4$ has dimension $\lceil \frac{4}{2}\rceil = 2$. We note that $|J_4| = 2$, and so each principal blow-up at $p \in \cF_4$ has two partial functions, one coming from the Pansu derivative of $\sqrt{|x_3|}$ and one coming from the Pansu derivative of $\sqrt[3]{|x_4|}$. Thus each principal blow-up has a 1-parameter family of translates, meaning $\dim \BU (\cF_4) = 3$.

For $n = 5$, the family of principal blow-ups along $\cF_5$ has dimension $\lceil \frac{5}{2}\rceil = 3$. Again we have $|J_5| = 2$, and so each principal blow-up at $p \in \cF_5$ has two partial functions. Thus each principal blow-up has a 1-parameter family of translates, meaning $\dim \BU (\cF_5) = 4$.

For $n \geq 6$, the family of principal blow-ups along $\cF_n$ has dimension $\lceil \frac{n}{2}\rceil$. We note that $|J_n| \geq 3$, and so each principal blow-up has at least 3 partial functions. Thus each principal blow-up has a 2-parameter family of translates, meaning $\dim \BU (\cF_n) = \lceil\frac{n}{2}\rceil +2$. For $n =6$ or $7$, this is equal to $n-1$, and for $n > 7$, this is strictly less than $n-1$.
\end{proof}

This finding begs for further exploration. It would be interesting to understand and classify which Carnot groups $G$ in general have horofunction boundaries of dimension $\dim G -1$. While higher Heisenberg groups of arbitrary dimension have follow our expectations, once we have higher nilpotency class, we see a change. Interestingly, the threshold between dimensions 7 and 8 which appears in the horofunction boundary of filiform Lie groups also appears in the classification of Carnot gradings on nilpotent Lie algebras \cite{cornulier-gradings}.

\begin{question}
For other infinite families of Carnot groups, is there a dimensional threshold at which the horofunction boundary ceases to be of the expected dimension? Is this threshold a function of the nilpotency class, the dimension of the first layer in the grading, and/or other properties of the group?
\end{question}

\bibliography{refs}

% \bib, bibdiv, biblist are defined by the amsrefs package.
\begin{bibdiv}
\begin{biblist}

\bib{azencott}{book}{
      author={Azencott, Robert},
       title={Espaces de {P}oisson des groupes localement compacts},
      series={Lecture Notes in Mathematics},
   publisher={Springer-Verlag, Berlin-New York},
        date={1970},
      volume={Vol. 148},
      review={\MR{501376}},
}

\bib{beer-convex}{book}{
      author={Beer, Gerald},
       title={Topologies on closed and closed convex sets},
      series={Mathematics and its Applications},
   publisher={Kluwer Academic Publishers Group, Dordrecht},
        date={1993},
      volume={268},
        ISBN={0-7923-2531-1},
         url={https://doi.org/10.1007/978-94-015-8149-3},
      review={\MR{1269778}},
}

\bib{breuillard-loc-cpct-gps}{article}{
      author={Breuillard, Emmanuel},
       title={Geometry of locally compact groups of polynomial growth and shape
  of large balls},
        date={2014},
        ISSN={1661-7207},
     journal={Groups Geom. Dyn.},
      volume={8},
      number={3},
       pages={669\ndash 732},
         url={https://doi-org.ezproxy.library.wisc.edu/10.4171/GGD/244},
      review={\MR{3267520}},
}

\bib{kramer}{article}{
      author={Ciobotaru, Corina},
      author={Kramer, Linus},
      author={Schwer, Petra},
       title={Polyhedral compactifications, i},
        date={2020},
     journal={arXiv preprint arXiv:2002.12422},
}

\bib{cornulier-gradings}{article}{
      author={Cornulier, Yves},
       title={Gradings on {L}ie algebras, systolic growth, and cohopfian
  properties of nilpotent groups},
        date={2016},
        ISSN={0037-9484,2102-622X},
     journal={Bull. Soc. Math. France},
      volume={144},
      number={4},
       pages={693\ndash 744},
         url={https://doi.org/10.24033/bsmf.2725},
      review={\MR{3562610}},
}

\bib{DF-stars}{article}{
      author={Duchin, Moon},
      author={Fisher, Nate},
       title={Stars at infinity in {T}eichm{\"u}ller space},
        date={2021},
     journal={Geometriae Dedicata},
       pages={1\ndash 15},
}

\bib{dymarz-fisher-xie}{article}{
      author={Dymarz, Tullia},
      author={Fisher, David},
      author={Xie, Xiangdong},
       title={A fibered {T}ukia theorem for nilpotent {L}ie groups},
        date={2023},
        ISSN={2737-0690,2737-114X},
     journal={Ann. Fenn. Math.},
      volume={48},
      number={2},
       pages={653\ndash 680},
         url={https://doi.org/10.54330/afm.137788},
      review={\MR{4652425}},
}

\bib{daniilidis}{article}{
      author={Daniilidis, A.},
      author={Garrido, M.~I.},
      author={Jaramillo, J.~A.},
      author={Tapia-Garc\'ia, S.},
       title={Horofunction extension and metric compactifications},
        date={2025},
        ISSN={2330-0000},
     journal={Trans. Amer. Math. Soc. Ser. B},
      volume={12},
       pages={1130\ndash 1155},
         url={https://doi.org/10.1090/btran/234},
      review={\MR{4952998}},
}

\bib{DM61}{article}{
      author={Dynkin, E.~B.},
      author={Maljutov, M.~B.},
       title={Random walk on groups with a finite number of generators},
        date={1961},
        ISSN={0002-3264},
     journal={Dokl. Akad. Nauk SSSR},
      volume={137},
       pages={1042\ndash 1045},
      review={\MR{131904}},
}

\bib{engelking}{book}{
      author={Engelking, Ryszard},
       title={General topology},
     edition={Second},
      series={Sigma Series in Pure Mathematics},
   publisher={Heldermann Verlag, Berlin},
        date={1989},
      volume={6},
        ISBN={3-88538-006-4},
        note={Translated from the Polish by the author},
      review={\MR{1039321}},
}

\bib{FNG}{article}{
      author={Fisher, Nate},
      author={Nicolussi~Golo, Sebastiano},
       title={Sub-{F}insler {H}orofunction {B}oundaries of the {H}eisenberg
  {G}roup},
        date={2021},
     journal={Anal. Geom. Metr. Spaces},
      volume={9},
      number={1},
       pages={19\ndash 52},
         url={https://doi.org/10.1515/agms-2020-0121},
      review={\MR{4237371}},
}

\bib{fischer-ruzhansky}{book}{
      author={Fischer, Veronique},
      author={Ruzhansky, Michael},
       title={Quantization on nilpotent {L}ie groups},
      series={Progress in Mathematics},
   publisher={Birkh\"{a}user/Springer, [Cham]},
        date={2016},
      volume={314},
        ISBN={978-3-319-29557-2; 978-3-319-29558-9},
  url={https://doi-org.ezproxy.library.wisc.edu/10.1007/978-3-319-29558-9},
      review={\MR{3469687}},
}

\bib{guivarch-quasinorm}{article}{
      author={Guivarc'h, Yves},
       title={Croissance polynomiale et p\'{e}riodes des fonctions
  harmoniques},
        date={1973},
        ISSN={0037-9484},
     journal={Bull. Soc. Math. France},
      volume={101},
       pages={333\ndash 379},
         url={http://www.numdam.org/item?id=BSMF_1973__101__333_0},
      review={\MR{369608}},
}

\bib{gutierrez}{inproceedings}{
      author={Guti{\'e}rrez, Armando~W},
       title={The horofunction boundary of finite-dimensional $\ell_p $
  spaces},
organization={Instytut Matematyczny Polskiej Akademii Nauk},
        date={2019},
   booktitle={Colloquium mathematicum},
      volume={155},
       pages={51\ndash 65},
}

\bib{ji-schilling-polyhedral}{article}{
      author={Ji, Lizhen},
      author={Schilling, Anna-Sofie},
       title={Polyhedral horofunction compactification as polyhedral ball},
        date={2016},
     journal={arXiv preprint arXiv:1607.00564},
}

\bib{karlsson-stars}{article}{
      author={Karlsson, Anders},
       title={On the dynamics of isometries},
        date={2005},
        ISSN={1465-3060},
     journal={Geom. Topol.},
      volume={9},
       pages={2359\ndash 2394 (electronic)},
      review={\MR{MR2209375 (2007h:22014)}},
}

\bib{Karlsson-hahnBanach}{article}{
      author={Karlsson, Anders},
       title={Hahn-{B}anach for metric functionals and horofunctions},
        date={2021},
        ISSN={0022-1236,1096-0783},
     journal={J. Funct. Anal.},
      volume={281},
      number={2},
       pages={Paper No. 109030, 17},
         url={https://doi.org/10.1016/j.jfa.2021.109030},
      review={\MR{4242963}},
}

\bib{karlsson-ledrappier}{article}{
      author={Karlsson, Anders},
      author={Ledrappier, Fran{\c{c}}ois},
       title={On laws of large numbers for random walks},
        date={2006},
        ISSN={0091-1798},
     journal={Ann. Probab.},
      volume={34},
      number={5},
       pages={1693\ndash 1706},
      review={\MR{MR2271477}},
}

\bib{karlsson-ledrappier-drift}{article}{
      author={Karlsson, Anders},
      author={Ledrappier, Fran\c~cois},
       title={Linear drift and {P}oisson boundary for random walks},
        date={2007},
        ISSN={1558-8599,1558-8602},
     journal={Pure Appl. Math. Q.},
      volume={3},
      number={4},
       pages={1027\ndash 1036},
         url={https://doi.org/10.4310/PAMQ.2007.v3.n4.a8},
      review={\MR{2402595}},
}

\bib{kapovich-leeb-finsler}{article}{
      author={Kapovich, Michael},
      author={Leeb, Bernhard},
       title={Finsler bordifications of symmetric and certain locally symmetric
  spaces},
        date={2018},
        ISSN={1465-3060},
     journal={Geom. Topol.},
      volume={22},
      number={5},
       pages={2533\ndash 2646},
         url={https://doi.org/10.2140/gt.2018.22.2533},
      review={\MR{3811766}},
}

\bib{karlsson-metz}{article}{
      author={Karlsson, Anders},
      author={Metz, Volker},
      author={Noskov, Gennady~A},
       title={Horoballs in simplices and {M}inkowski spaces},
        date={2006},
     journal={International journal of mathematics and mathematical sciences},
      volume={2006},
}

\bib{KN-cc}{article}{
      author={Klein, Tom},
      author={Nicas, Andrew},
       title={The horofunction boundary of the {H}eisenberg group: the
  {C}arnot-{C}arath{\'e}odory metric},
        date={2010},
     journal={Conformal Geometry and Dynamics of the American Mathematical
  Society},
      volume={14},
      number={15},
       pages={269\ndash 295},
}

\bib{ledonne-primer}{article}{
      author={Le~Donne, Enrico},
       title={A primer on {C}arnot groups: homogenous groups,
  {C}arnot-{C}arath\'{e}odory spaces, and regularity of their isometries},
        date={2017},
     journal={Anal. Geom. Metr. Spaces},
      volume={5},
      number={1},
       pages={116\ndash 137},
         url={https://doi.org/10.1515/agms-2017-0007},
      review={\MR{3742567}},
}

\bib{asympt-horobdry}{article}{
      author={Le~Donne, Enrico},
      author={Nicolussi~Golo, Sebastiano},
      author={Sambusetti, Andrea},
       title={Asymptotic behavior of the {R}iemannian {H}eisenberg group and
  its horoboundary},
        date={2017},
        ISSN={0373-3114},
     journal={Ann. Mat. Pura Appl. (4)},
      volume={196},
      number={4},
       pages={1251\ndash 1272},
         url={https://doi.org/10.1007/s10231-016-0615-2},
      review={\MR{3673666}},
}

\bib{MR3943324}{article}{
      author={Le~Donne, Enrico},
      author={Rigot, S\'{e}verine},
       title={Besicovitch covering property on graded groups and applications
  to measure differentiation},
        date={2019},
        ISSN={0075-4102},
     journal={J. Reine Angew. Math.},
      volume={750},
       pages={241\ndash 297},
         url={https://doi.org/10.1515/crelle-2016-0051},
      review={\MR{3943324}},
}

\bib{muger-notes}{misc}{
      author={M\"uger, Michael},
       title={{N}otes on the theorem of
  {B}aker-{C}ampbell-{H}ausdorff-{D}ynkin},
         how={{H}omepage of {M}ichael {M}\"uger (Online)},
        date={2020},
         url={https://www.math.ru.nl/~mueger/PDF/BCHD.pdf},
}

\bib{maher-tiozzo}{article}{
      author={Maher, Joseph},
      author={Tiozzo, Giulio},
       title={Random walks on weakly hyperbolic groups},
        date={2018},
        ISSN={0075-4102},
     journal={J. Reine Angew. Math.},
      volume={742},
       pages={187\ndash 239},
         url={https://doi.org/10.1515/crelle-2015-0076},
      review={\MR{3849626}},
}

\bib{rockafellar-convex}{book}{
      author={Rockafellar, Ralph~Tyrell},
       title={Convex analysis},
   publisher={Princeton University Press},
     address={Princeton},
        date={1970},
        ISBN={9781400873173},
         url={https://doi.org/10.1515/9781400873173},
}

\bib{schilling-thesis}{thesis}{
      author={Schilling, Anna-Sofie},
       title={The horofunction compactification of finite-dimensional normed
  vector spaces and of symmetric spaces},
        type={Ph.D. Thesis},
        date={2021},
}

\bib{walsh-norm}{article}{
      author={Walsh, Cormac},
       title={The horofunction boundary of finite-dimensional normed spaces},
        date={2007},
        ISSN={0305-0041},
     journal={Math. Proc. Cambridge Philos. Soc.},
      volume={142},
      number={3},
       pages={497\ndash 507},
         url={https://doi.org/10.1017/S0305004107000096},
      review={\MR{2329698}},
}

\end{biblist}
\end{bibdiv}

\end{document}